\documentclass{article}

\parskip \medskipamount
\usepackage[T1]{fontenc}
\usepackage{geometry}
\usepackage{titling}
\usepackage{setspace}
\usepackage{xcolor}
\usepackage{authblk}
\usepackage{amsmath,amssymb,amsthm,amsfonts,bm,bbm}

\DeclareMathOperator*{\argmax}{arg\,max}

\usepackage{mathrsfs,stmaryrd}
\usepackage{upgreek}
\usepackage{stmaryrd}
\SetSymbolFont{stmry}{bold}{U}{stmry}{m}{n}
\usepackage{indentfirst}
\usepackage{subdepth}
\usepackage{graphicx}
\usepackage{diagbox}
\usepackage{pdfpages}
\usepackage[latin1]{inputenc}
\usepackage{url}
\usepackage{float}
\usepackage{appendix}
\usepackage{etoolbox}
\usepackage{microtype}
\usepackage[centerlast]{caption}
\usepackage[plainpages=false,hyperfootnotes=false]{hyperref}
\hypersetup{colorlinks=true,urlcolor=blue,linkcolor=red,citecolor=blue}
\usepackage{cleveref}
\crefname{equation}{\hspace{-0.4em}}{\hspace{-0.4em}}
\crefrangeformat{equation}{(#3#1#4)--(#5#2#6)}
\usepackage{bookmark}
\usepackage{framed}
\usepackage{enumitem}
\usepackage{apacite}

\newtheorem{theorem}{Theorem}[section]
\newtheorem{lemma}[theorem]{Lemma}

\newtheorem{remark}[theorem]{Remark}
\newtheorem{definition}[theorem]{Definition}

\newtheorem{example}[theorem]{Example}
\renewenvironment{proof}{\noindent {\bf Proof.}}{\hfill $\Box$}
\allowdisplaybreaks[4]

\hypersetup{
  pdftitle={On strong equilibria for time-inconsistent control problems with higher-order moments},
  pdfauthor={Y. Wang}
}

\title{Closed-loop strong equilibria for time-inconsistent control problems with higher-order moments
       \thanks{This work is supported by the National Natural Science Foundation of China under Grant 12401611 and the CTBU Research Projects under Grant 2355010.}
}
\date{\vspace{-6ex}}
\author{
    Yike Wang\thanks{School of Finance, Chongqing Technology and Business University, Chongqing, China}.
}

\begin{document}
\maketitle

\begin{abstract}
In this paper, we study closed-loop strong equilibrium strategies for the time-inconsistent control problem with higher-order moments formulated by [Wang et al. SIAM J. Control. Optim., 63 (2025), 1560--1589].
Since time-inconsistency makes the dynamic programming principle inapplicable, the problem is treated as a game between the decision maker and her future selves.
As a time-consistent solution, the previously proposed Nash equilibrium control is merely a stationary point and does not necessarily reach a maximum in the game-theoretical prospective.
To address this issue, we consider the strong equilibrium strategy, from which any spike deviation will be worse off.
We derive sufficient conditions for strong equilibrium strategies by expanding the objective function corresponding to the spike deviation with respect to the variational factor up to higher orders,
and simplify the result by exploiting the partial differential equations that a Nash equilibrium control satisfies.
Next, we examine whether the derived Nash equilibrium control is a strong equilibrium strategy.
In particular, we provide a sufficient condition for the case with linear controlled equations.
When the model parameters are constant, we find that the derived Nash equilibrium control for the mean-variance problem must be a strong equilibrium strategy,
while that for the mean-variance-skewness-kurtosis problem is not necessarily a strong equilibrium strategy.
\end{abstract}

\vspace{3mm}

\vspace{3mm}

\section{Introduction}

Bellman's dynamic programming principle is recognized as a cornerstone in the field of control and optimization.
However, there still exist many control problems that violate the dynamic programming principle, where an optimal decision made at a future epoch differs from the initial optimal strategy.
The deviation phenomenon is called ``time-inconsistency'', for which studies can be traced back to \cite{Strotz-1955,Pollak-1968,Peleg-Yaari-1973,Goldman-1980,Laibson-1997,Laibson-1998}.
Akin to the non-exponential discounting problem and the mean-variance problem (see, e.g., \cite{Ekeland-Mbodji-Pirvu-2012} and \cite{Bjork-Murgoci-Zhou-2014}),
these time-inconsistent control problems have attracted much attention and have been studied for several years.
It has become common to treat a time-inconsistent problem as a game between the decision maker and her/his future selves, for which a subgame perfect Nash equilibrium point is adopted as a time-consistent solution.

In the continuous-time framework, deviations from the equilibrium point may occur within an arbitrarily short period.
Notably, the length of this period is usually referred to as the ``variational factor'' (denoted by $\varepsilon $ in this paper).
Given the continuity of the objective function and state processes, the effect of a series of deviations tends to vanish as the variational factor approaches zero.
Due to this situation, in addition to the multi-person hierarchical game approximation (see \cite{Yong-2012-MCRF,Yong-2012-AMASES,Bjork-Murgoci-2014}),
researchers have considered a weaker definition for Nash equilibrium control:
the ratio of the difference in the objective function led by a deviation to the variational factor is not positive, provided that the decision maker intends to maximize the objective function.
This definition is similar to employing the first-order derivative to describe the growth trend.
See, e.g., \cite{Hu-Jin-Zhou-2012,Hu-Jin-Zhou-2017} for the open-loop control problem (where a deviation is imposed on the control process)
and \cite{Bjork-Khapko-Murgoci-2017,Wang-Zheng-2021} for the closed-loop control problem (where a deviation is imposed on the feedback mapping).
As a result, any deviation from the Nash equilibrium control will not lead to significant profits.

However, as mentioned in \cite[Remark 3.5]{Bjork-Khapko-Murgoci-2017}, the previous definition of Nash equilibrium control merely corresponds to a stationary point rather than a maximum point.
In other words, leading to no significant profit does not mean that this deviation will be worse off.
If a tiny profit that at most equals a higher-order infinitesimal of the variational factor exists, then it still motivates an economical or greedy person to deviate from the Nash equilibrium control.
To address this issue, \cite{He-Jiang-2021} greatly extended work of \cite{Huang-Zhou-2021} to fit most time-inconsistent problems and defined strong and regular equilibrium strategies.
Any (resp. a concerned) deviation from a strong (resp. regular) equilibrium strategy is not profitable.

Inspired by the notion of \cite{He-Jiang-2021}, in the present paper we study the strong equilibria for time-inconsistent control problems with higher-order moments.
In fact, in \cite{Wang-Liu-Bensoussan-Yiu-Wei-2025} we have studied the Nash equilibrium controls for those problems formulated by
\begin{equation*}
\left\{ \begin{aligned}
& J ( t,u ) = \mathbb{E}_{t} \Big[ \Psi \Big( t, {X}^{u}_{t}, {X}^{u}_{T}, \mathbb{E}_{t} [ {X}^{u}_{T} ], \mathbb{E}_{t} \big[ ( {X}^{u}_{T} - \mathbb{E}_{t} [ {X}^{u}_{T} ] )^{2} \big], 
                                                                                                   \ldots, \mathbb{E}_{t} \big[ ( {X}^{u}_{T} - \mathbb{E}_{t} [ {X}^{u}_{T} ] )^{n} \big] \Big) \Big], \\
& d {X}^{u}_{t} = b ( t, {X}^{u}_{t}, {u}_{t} ) dt + \sigma ( t, {X}^{u}_{t}, {u}_{t} ) d {W}_{t},
\end{aligned} \right.
\end{equation*}
where the superscript emphasizes the dependence of the control $u$ for the state process $X$.
The (semi-)analytical solution can be obtained in the following particular case:
\begin{equation}
\left\{ \begin{aligned}
& J ( t,u ) = \mathbb{E}_{t} [ {X}^{u}_{T} ] + \mathbb{E}_{t} \Big[ \psi \Big( t, \mathbb{E}_{t} \big[ ( {X}^{u}_{T} - \mathbb{E}_{t} [ {X}^{u}_{T} ] )^{2} \big], 
                                                                          \ldots, \mathbb{E}_{t} \big[ ( {X}^{u}_{T} - \mathbb{E}_{t} [ {X}^{u}_{T} ] )^{n} \big] \Big) \Big], \\
& d {X}^{u}_{t} = ( {A}_{t} {X}^{u}_{t} + {B}_{t} {u}_{t} + {C}_{t} ) dt + ( {D}_{t} {u}_{t} + {F}_{t} ) d {W}_{t}.
\end{aligned} \right.
\label{eq: linear control problem :eq}%%
\end{equation}
The results of this particular case are summarized in \Cref{exp: linear control problem} in the present paper, and are extended for the notion of strong equilibrium in \Cref{sec: Example continuation}.
The higher-order moment problem is not only a straightforward extension of the mean-variance problem, but also may heuristically produce a solution for control problems with general utility functions or non-linear law-dependent preferences.
See, e.g., \cite[Example 4.6 and Section 5.3]{Wang-Liu-Bensoussan-Yiu-Wei-2025}.
Since these control problems are considered in many mathematical finance scenarios, such as portfolio selection and intertemporal consumption planning,
it makes sense to investigate whether the previously derived Nash equilibrium controls are strong equilibrium strategies.
The results indicate whether those controls can be exactly implemented in a stable and continuous manner by a sophisticated person in the game-theoretical framework.

The main contributions of the present paper are as follows.
\emph{In theory}, we provide the sufficient condition under which the closed-loop Nash equilibrium control (CNEC) 
for our higher-order moment problem is a closed-loop strong equilibrium strategy (CSES)
by expanding the objective function corresponding to the deviation w.r.t. the variation factor up to adequate orders.
The sufficient condition provided in the present paper is much simpler than that derived in \cite{He-Jiang-2021} since we exploit the partial differential equations (PDEs) that the CNEC satisfies at the beginning of the process of expanding the objective function;
see \Cref{thm: second-order derivative condition for CSES,thm: third-order derivative condition for CSES} and their proofs.
In particular, we find that the third-order expansion is needed if the partial derivative w.r.t. the state argument of the CNEC vanishes at some points.
This interesting proposition is not reported in \cite{He-Jiang-2021}.
\emph{In practice}, we investigate the particular case formulated by \cref{eq: linear control problem :eq} and provide sufficient conditions for the CSES.
On the one hand, we find that the derived CNEC for \cref{eq: linear control problem :eq} is a CSES if $F$ is propositional to $D$ and almost every $| {A}_{t} |$ is far away from some points determined by $( B,D )$.
On the other hand, we consider the mean-variance and mean-variance-skewness-kurtosis problems.
Provided that $( A,B,D,F )$ are constant, we find that the derived mean-variance CNEC must be a CSES, but the derived mean-variance-skewness-kurtosis CNEC is not necessarily a CSES.

The rest of this paper is organized as follows.
In \Cref{sec: Model and problem formulation}, we formulate the control problem in the closed-loop type and introduce the definition of the closed-loop strong equilibrium strategy.
In \Cref{sec: Characterization of CNEC revisited}, we revisit the main results for closed-loop Nash equilibrium control.
In \Cref{sec: Characterization of CSES}, we characterize the closed-loop strong equilibrium strategy according to the sufficient equilibrium conditions up to the third order.
In \Cref{sec: Example continuation}, we provide the CSES results for \cref{eq: linear control problem :eq}.
Finally, we provide brief concluding remarks in \Cref{sec: Concluding remarks}.

\section{Model and problem formulation}
\label{sec: Model and problem formulation}

Let $T$ be a fixed finite time horizon, and $( \Omega, \mathcal{F}, \mathbb{F}, \mathbb{P} )$ with $\mathbb{F} := \{ \mathcal{F}_{t} \}_{ t \ge 0 }$ be a filtered probability space satisfying the usual condition.
For notational simplicity, we consider the one-dimensional standard Brownian motion $\{ {W}_{t} \}_{ t \ge 0 }$, suppose that the filtration $\mathbb{F}$ is generated by $W$ and write $\mathbb{E}_{t} [ \cdot ] := \mathbb{E} [ \cdot | \mathcal{F}_{t} ]$.
In addition, with a slight abuse of notation, we let ${C}^{ {k}_{1}, {k}_{2}, \ldots }$ denote the set of all the functions that are continuously differentiable in the $j$-th argument up to the ${k}_{j}$-th order.

Let us introduce the following controlled stochastic differential equation (SDE):
\begin{equation}
\left\{ \begin{aligned}
d {X}^{t,x,u}_{s} & = b \big( s, {X}^{t,x,u}_{s}, u ( s, {X}^{t,x,u}_{s} ) \big) ds + \sigma \big( s, {X}^{t,x,u}_{s}, u ( s, {X}^{t,x,u}_{s} ) \big) d {W}_{s}, \quad \forall s \in [ t,T ]; \\
  {X}^{t,x,u}_{t} & = x,
\end{aligned} \right.
\label{eq: controlled SDE :eq}%%
\end{equation}
where the functions $b, \sigma: [ 0,T ] \times \mathbb{R} \times \mathbb{R} \to \mathbb{R}$ are of ${C}^{2,4,4}$.
The superscript emphasizes the dependence of the initial pair $( t,x )$ (namely, the initial condition ${X}_{t} = x$) and the feedback control $u: [ 0,T ] \times \mathbb{R} \to \mathbb{R}$ for the state process $X$ on $[ t,T ] \times \Omega $.
On the other hand, we consider the following objective function with finitely many higher-order original moments ($n \in \mathbb{N}$, $n \ge 2$):
\begin{equation}
J ( t,x,u ) := \mathbb{E}_{t} \Big[ \Phi \Big( t, x, {X}^{t,x,u}_{T}, \mathbb{E}_{t} [ {X}^{t,x,u}_{T} ], \mathbb{E}_{t} [ ( {X}^{t,x,u}_{T} )^{2} ], \ldots, \mathbb{E}_{t} [ ( {X}^{t,x,u}_{T} )^{n} ] \Big) \Big],
\label{eq: objective function :eq}%%
\end{equation}
where the function $\Phi: [ 0,T ] \times \mathbb{R} \times \mathbb{R} \times \mathbb{R}^{n} \to \mathbb{R}$ is of ${C}^{0,0,4,2}$.
Clearly, provided that \cref{eq: controlled SDE :eq} admits a unique $\mathbb{F}$-adapted strong solution for any initial pair $( t,x )$,
which leads to ${X}^{ t, {X}^{ 0,x,u }_{t}, u }_{s} = {X}^{0,x,u}_{s}$ for any $t \in [ 0,s ] \subseteq [ 0,T ]$ and $x \in \mathbb{R}$, plugging
\begin{equation*}
\Phi ( t, x, y, {z}_{1}, \ldots, {z}_{n} ) = \Psi \bigg( t,x,y, {z}_{1}, {z}_{2} - {z}_{1}^{2}, \ldots, \sum_{j=0}^{n} \binom{n}{j} (-1)^{j} {z}_{1}^{j} {z}_{n-j} \bigg)
\end{equation*}
with $x = {X}^{ 0, {X}_{0}, u }_{t}$ into \cref{eq: objective function :eq} immediately produces
\begin{equation*}
J ( t, {X}^{ 0, {X}_{0}, u }_{t}, u ) = \mathbb{E}_{t} \Big[ \Psi \Big( t, {X}^{u}_{t}, {X}^{u}_{T}, \mathbb{E}_{t} [ {X}^{u}_{T} ], \mathbb{E}_{t} \big[ ( {X}^{u}_{T} - \mathbb{E}_{t} [ {X}^{u}_{T} ] )^{2} \big], 
                                                                                                   \ldots, \mathbb{E}_{t} \big[ ( {X}^{u}_{T} - \mathbb{E}_{t} [ {X}^{u}_{T} ] )^{n} \big] \Big) \Big].
\end{equation*}
For notational simplicity, we introduce the row vector $\vec{m}^{u} ( t,x ) \in \mathbb{R}^{n}$ with its $j$-th element being $\mathbb{E}_{t} [ ( {X}^{t,x,u}_{T} )^{j} ]$
so that $J ( t,x,u ) = \mathbb{E}_{t} [ \Phi ( t, x, {X}^{t,x,u}_{T}, \vec{m}^{u} ( t,x ) ) ]$.

\begin{definition}\label{def: admissible control}
A measurable function $u: [ 0,T ] \times \mathbb{R} \to \mathbb{R}$ is an admissible feedback control if for any initial pair $( t,x )$,
\cref{eq: controlled SDE :eq} admits a unique $\mathbb{F}$-adapted strong solution such that $\mathbb{E} [ | {X}^{t,x,u}_{T} |^{n} ] < \infty $ and $\mathbb{E} [ | \Phi ( t, x, {X}^{t,x,u}_{T}, \vec{m}^{u} ( t,x ) ) | ] < \infty $.
\end{definition}

Let $\mathcal{U}$ denote the set of all the admissible feedback controls.
Following \cite{Ekeland-Pirvu-2008,Ekeland-Lazrak-2010,Ekeland-Mbodji-Pirvu-2012,Bjork-Murgoci-2014,Bjork-Murgoci-Zhou-2014,Bjork-Khapko-Murgoci-2017},
it is supposed to find a CNEC as a time-consistent solution, which is defined as follows:

\begin{definition}\label{def: closed-loop Nash equilibrium control}
$\tilde{u} \in \mathcal{U}$ is a CNEC if
\begin{equation}
\left\{ \begin{aligned}
& 0 \le \liminf_{ \varepsilon \downarrow 0 }
        \frac{1}{\varepsilon } \big( J ( t, {X}^{ 0, {X}_{0}, \tilde{u} }_{t}, \tilde{u} ) - J ( t, {X}^{ 0, {X}_{0}, \tilde{u} }_{t}, \tilde{u}^{ t, \varepsilon, \zeta } ) \big), \\
& \forall \zeta \in \liminf_{ \varepsilon \downarrow 0 } \mathcal{Z}^{ t, \tilde{u} }_{\varepsilon }, ~ \mathbb{P}-a.s., ~ a.e. ~ t \in [ 0,T ),
\end{aligned} \right.
\label{eq: closed-loop Nash equilibrium control :eq}%%
\end{equation}
where $\tilde{u}^{ t, \varepsilon, \zeta }$ is expressed as $\tilde{u}^{ t, \varepsilon, \zeta } ( s, \cdot ) = \tilde{u} ( s, \cdot ) {1}_{\{ s \notin [ t, t + \varepsilon ) \}} + \zeta {1}_{\{ s \in [ t, t + \varepsilon ) \}}$,
and $\mathcal{Z}^{ t, \tilde{u} }_{\varepsilon } := \{ \zeta \in \mathbb{R}: \tilde{u}^{ t, \varepsilon, \zeta } \in \mathcal{U} \}$.
\end{definition}

\begin{remark}
If $b ( t, \cdot, \cdot )$, $\sigma ( t, \cdot, \cdot )$ and $\tilde{u} ( t, \cdot )$ are Lipschitz, for which the Lipschitz continuity parameters are uniformly bounded for all $t \in [ 0,T ]$,
and $b ( \cdot, 0, 0 )$, $\sigma ( \cdot, 0, 0 )$ and $\tilde{u} ( \cdot, 0 )$ are bounded on $[ 0,T ]$,
then \cref{eq: controlled SDE :eq} is a Lipschitz SDE corresponding to $u = \tilde{u}, \tilde{u}^{ t, \varepsilon, \zeta }$ (for any $( \varepsilon, \zeta ) \in ( 0,T-t ] \times \mathbb{R}$) with
\begin{equation*}
\int_{t}^{T} \mathbb{E} \big[ \big| b \big( s,0, u ( s,0 ) \big) \big|^{p} + \big| \sigma \big( s,0, u ( s,0 ) \big) \big|^{p} \big] ds < \infty, \quad \forall p \ge 2,
\end{equation*}
implying that $\mathbb{E} [ \sup_{ s \in [ t,T ] } | {X}^{ t,x,u }_{s} |^{p} ] < \infty $ due to \cite[Theorem 16.1.2]{Cohen-Elliott-2015}.
Furthermore, if $\Phi ( t,x,y, \vec{z} )$ exhibits polynomial growth in $y \in \mathbb{R}$, then $\tilde{u}, \tilde{u}^{ t, \varepsilon, \zeta } \in \mathcal{U}$.
Hence, $\liminf_{ \varepsilon \downarrow 0 } \mathcal{Z}^{ t, \tilde{u} }_{\varepsilon } = \mathbb{R}$.
\end{remark}

In fact, there are various definitions for the CNEC that are slightly stronger than \Cref{def: closed-loop Nash equilibrium control}.
For example, \cref{eq: closed-loop Nash equilibrium control :eq} could be replaced by
\begin{equation}
\left\{ \begin{aligned}
& 0 \le \liminf_{\varepsilon \downarrow 0}
        \frac{1}{\varepsilon } \big( J ( t,x, \tilde{u} ) - J ( t,x, \tilde{u}^{ t, \varepsilon, \zeta } ) \big), \\
& \forall \zeta \in \liminf_{ \varepsilon \downarrow 0 } \mathcal{Z}^{ t, \tilde{u} }_{\varepsilon }, ~ \mathbb{P}-a.s., ~ \forall ( t,x ) \in [ 0,T ) \times \mathbb{R};
\end{aligned} \right.
\label{eq: stronger definition for CNEC :eq}%%
\end{equation}
or the perturbation parameter $\zeta $ could be allowed to be a general function, i.e., $\tilde{u}^{ t, \varepsilon, \zeta } ( s, \cdot ) = \zeta ( s, \cdot )$ for $s \in [ t, t + \varepsilon )$.
In a nutshell, a CNEC is supposed to be an equilibrium point such that any deviation from it (i.e., changing $\tilde{u} |_{ [ t, t + \varepsilon ) \times \mathbb{R} }$ to $\zeta $) does not lead to any significant profit.

Notably, a deviation for a CNEC as defined in \Cref{def: closed-loop Nash equilibrium control} is not necessarily worse off.
Indeed, the sufficient condition for the CNEC studied in, e.g., \cite{Bjork-Khapko-Murgoci-2017} and \cite{Wang-Liu-Bensoussan-Yiu-Wei-2025},
only guarantees $J ( t,x, \tilde{u} ) - J ( t,x, \tilde{u}^{ t, \varepsilon, \zeta } ) \ge o ( \varepsilon )$ rather than $J ( t,x, \tilde{u} ) \ge J ( t,x, \tilde{u}^{ t, \varepsilon, \zeta } )$.
For example, if $J ( t,x, \tilde{u}^{ t, \varepsilon, \zeta } ) - J ( t,x, \tilde{u} ) = K ( t,x, \varepsilon ) {\varepsilon }^{2} > 0$ holds for all sufficiently small $\varepsilon $ and some bounded function $K$,
then $\tilde{u}$ is a CNEC, but the deviation $\tilde{u}^{ t, \varepsilon, \zeta }$ is profitable.
This result contradicts the original idea of Nash equilibrium control as described in, e.g., \cite{Bjork-Murgoci-Zhou-2014}.
To address this issue, \cite{He-Jiang-2021} studied strong equilibrium strategies for time-inconsistent control problems, including the mean-variance problem and the nonexponential discounting problem,
and found that whether a CNEC is a strong equilibrium strategy depends on the range of $\zeta $.
If $\zeta $ is allowed to be a general function, then the CNECs derived for the mean-variance problem and the non-exponential discounting problem are not strong equilibrium strategies.
In particular, the derived CNEC for the mean-variance problem is a strong equilibrium strategy if $\zeta $ is only assumed to be a constant.
Therefore, in our higher-order moment problem, we consider only the constant $\zeta $, and the strong equilibrium strategy is defined as follows.

\begin{definition}\label{def: closed-loop strong equilibrium strategy}
$\tilde{u} \in \mathcal{U}$ is a closed-loop strong equilibrium strategy (CSES), if
\begin{equation}
\left\{ \begin{aligned}
& J ( t, {X}^{ 0, {X}_{0}, \tilde{u} }_{t}, \tilde{u} ) \ge J ( t, {X}^{ 0, {X}_{0}, \tilde{u} }_{t}, \tilde{u}^{ t, \varepsilon, \zeta } ), \\
& \forall \varepsilon \in ( 0, {\epsilon }_{ t, \zeta } ), ~ \exists {\epsilon }_{ t, \zeta } \in ( 0, T-t ), ~ \forall \zeta \in \liminf_{ \varepsilon \downarrow 0 } \mathcal{Z}^{ t, \tilde{u} }_{\varepsilon },
  ~ \mathbb{P}-a.s., ~ a.e. ~ t \in [ 0,T ).
\end{aligned} \right.
\label{eq: closed-loop strong equilibrium strategy :eq}%%
\end{equation}
\end{definition}

\Cref{def: closed-loop strong equilibrium strategy} is weaker than the definition of strong equilibria in \cite{He-Jiang-2021} but stronger than the definition of regular equilibria therein.
In fact, a regular equilibrium strategy $\tilde{u}$ as defined in that work only realizes the optimum against $\tilde{u}^{ t, \varepsilon, \tilde{u} ( t, {X}_{t} ) }$
rather than $\tilde{u}^{ t, \varepsilon, \zeta }$ with an arbitrariness of $\zeta $.
Nevertheless, in terms of the result, the sufficient condition for a CSES in this paper also serves as a regular equilibrium strategy,
since the second-order derivative condition \cref{eq: second-order derivative condition :eq} for the equilibrium is included.

Clearly, a CSES must be a CNEC.
Therefore, it is supposed to first seek a CNEC and then to examine whether the derived CNEC is a CSES.
Notably, in line with the strong equilibrium strategies defined in \cite{He-Jiang-2021},
\Cref{def: closed-loop strong equilibrium strategy} does not require the uniformity of $\varepsilon $ with respect to $\zeta $.
In other words, a stronger definition of a CSES should consider
\begin{equation*}
\left\{ \begin{aligned}
& J ( t, {X}^{ 0, {X}_{0}, \tilde{u} }_{t}, \tilde{u} ) \ge J ( t, {X}^{ 0, {X}_{0}, \tilde{u} }_{t}, \tilde{u}^{ t, \varepsilon, \zeta } ), \\
& \forall \zeta \in \mathcal{Z}^{ t, \tilde{u} }_{\varepsilon }, ~ \forall \varepsilon \in ( 0, {\epsilon }_{t} ), ~ \exists {\epsilon }_{t} \in ( 0, T-t ), ~ \mathbb{P}-a.s., ~ a.e. ~ t \in [ 0,T ).
\end{aligned} \right.
\end{equation*}
This definition serves a discrete-time approximation with some time lattice, which is fixed before imposing the perturbations.
However, in the current asymptotic analysis for the CNEC and CSES, the abovementioned uniformity cannot be guaranteed by sufficient conditions unless additional complicated conditions are imposed.

To end this section, we list some notations that are adopted in the rest of this paper for ease of reference.
Corresponding to the function $\vec{F}^{ s,y, \vec{z} } := ( {U}^{ s,y, \vec{z} }, \vec{m} ): [ 0,T ] \times \mathbb{R} \to \mathbb{R}^{n+1}$, which is indexed by the parameter triplet $( s,y, \vec{z} ) \in [ 0,T ) \times \mathbb{R} \times \mathbb{R}^{n}$,
we let
\begin{equation*}
\vec{H}^{ s,y, \vec{z} } ( t,x,u ) := \vec{F}^{ s,y, \vec{z} }_{x} ( t,x ) b ( t,x,u ) + \vec{F}^{ s,y, \vec{z} }_{xx} ( t,x ) \hat{\sigma } ( t,x,u ).
\end{equation*}
For ${U}^{ s,y, \cdot } ( t,x ): \mathbb{R}^{n} \to \mathbb{R}$ with adequate smoothness,
its (transposed) gradient vector and Hessian matrix are denoted by $\vec{\lambda }^{ s,y, \vec{z} } ( t,x )$ and ${\mu }^{ s,y, \vec{z} } ( t,x )$, respectively.
Let $\langle \cdot, \cdot \rangle$ be the inner product of two vectors, $\| \cdot \|$ be the Euclidean norm, and $| \cdot |$ be the $\mathbb{L}^{1}$-norm (namely, the Manhattan norm).
For the perturbation $\zeta \in \mathbb{R}$, we define $\hat{\sigma } ( t,x, \zeta ) := \frac{1}{2} | {\sigma }^{\zeta } ( t,x, \zeta ) |^{2}$ and ${g}^{\zeta } ( t,x ) := g ( t,x, \zeta )$ 
and introduce the infinitesimal operator $\mathcal{D}_{\zeta }$ as follows:
\begin{equation*}
\mathcal{D}_{\zeta } f ( t,x ) = {f}_{t} ( t,x ) + {f}_{x} ( t,x ) {b}^{\zeta } ( t,x ) + {f}_{xx} ( t,x ) \hat{\sigma }^{\zeta } ( t,x ).
\end{equation*}
Furthermore, by iterating with $\mathcal{D}_{\zeta }^{0}$ as the identity operator (i.e., $\mathcal{D}_{\zeta }^{0} f ( t,x ) = f ( t,x )$),
we introduce the operators $\mathcal{D}_{\zeta }^{k}$ and $\mathscr{D}_{\zeta }^{k}$ for all positive integers $k$
with $\mathcal{D}_{\zeta }^{k} f ( t,x ) = \mathcal{D}_{\zeta } {g}^{\zeta } ( t,x )$ and $\mathscr{D}_{\zeta }^{k} f ( t,x ) = {g}^{\zeta }_{x} ( t,x ) {\sigma }^{\zeta } ( t,x )$, where ${g}^{\zeta } ( t,x ) = \mathcal{D}_{\zeta }^{k-1} f ( t,x )$.
To avoid misunderstandings, we should clarify that
\begin{align*}
& \mathcal{D}_{\zeta }^{k} {f}^{ g ( t,x ) } ( v,z ) = \big( \mathcal{D}_{\zeta }^{k} {f}^{h} ( t,x ) |_{ ( t,x ) = ( v,z ) } \big) \big|_{ h = g ( t,x ) }, \\
& \mathscr{D}_{\zeta }^{k} {f}^{ g ( t,x ) } ( v,z ) = \big( \mathscr{D}_{\zeta }^{k} {f}^{h} ( t,x ) |_{ ( t,x ) = ( v,z ) } \big) \big|_{ h = g ( t,x ) },
\end{align*}
for the $h$-indexed function ${f}^{h}$ and generic function $g$.
In addition, we define ${\partial }_{\xi } := \frac{\partial }{\partial \xi }$ and ${\partial }_{ \vec{z} } := ( {\partial }_{ {z}_{1} }, \ldots, {\partial }_{ {z}_{n} } )$ 
and introduce the operator $\mathcal{A}_{\xi }$ by $\mathcal{A}_{\xi } f ( \xi ) = f ( \tilde{u} ( t,x ) )$ for a given $\tilde{u} \in {C}^{1,2}$
so that applying the chain rule yields
\begin{align*}
    \mathcal{D}_{\zeta } \mathcal{A}_{\xi }
& = \mathcal{A}_{\xi } \mathcal{D}_{\zeta } 
  + ( \mathcal{D}_{\zeta } \tilde{u} ) \mathcal{A}_{\xi } {\partial }_{\xi } 
  + 2 \tilde{u}_{x} \hat{\sigma }^{\zeta } \mathcal{A}_{\xi } {\partial }_{x} {\partial }_{\xi } 
  + | \tilde{u}_{x} |^{2} \hat{\sigma }^{\zeta } \mathcal{A}_{\xi } {\partial }_{\xi }^{2} \\
& = \mathcal{A}_{\xi } \mathcal{D}_{\zeta } 
  + ( \mathcal{D}_{\zeta } \tilde{u} ) \mathcal{A}_{\xi } {\partial }_{\xi } 
  + 2 \tilde{u}_{x} \hat{\sigma }^{\zeta } {\partial }_{x} \mathcal{A}_{\xi } {\partial }_{\xi } 
  - | \tilde{u}_{x} |^{2} \hat{\sigma }^{\zeta } \mathcal{A}_{\xi } {\partial }_{\xi }^{2} \\
& = \mathcal{A}_{\xi } \mathcal{D}_{\zeta } 
  + \mathcal{D}_{\zeta } ( \tilde{u} \mathcal{A}_{\xi } {\partial }_{\xi } )
  - \tilde{u} \mathcal{D}_{\zeta } \mathcal{A}_{\xi } {\partial }_{\xi }
  - | \tilde{u}_{x} |^{2} \hat{\sigma }^{\zeta } \mathcal{A}_{\xi } {\partial }_{\xi }^{2},
\end{align*}
which will be used to simplify the derivations in \Cref{sec: Characterization of CSES}.

\section{Characterization of CNEC revisited}
\label{sec: Characterization of CNEC revisited}

The sufficient condition for the CNECs of our higher-order moment problem have been studied in \cite{Wang-Liu-Bensoussan-Yiu-Wei-2025}.
Here, we summarize the main results in the following theorem and omit the proof, for which interested reader can also refer to \Cref{lem: second-order perturbation} in the present paper.

\begin{theorem}\label{thm: sufficient condition for CNEC}
Suppose that $\tilde{u} \in \mathcal{U}$.
Assume that for every $( s,y, \vec{z} ) \in [ 0,T ) \times \mathbb{R} \times \mathbb{R}^{n}$,
$\vec{F}^{ s,y, \vec{z} } = ( {U}^{ s,y, \vec{z} }, \vec{m} )$ is the classic ${C}^{1,2}$-solution for the following PDE on $[ 0,T ] \times \mathbb{R}$:
\begin{equation}
0 = \mathcal{D}_{ \tilde{u} ( t,x ) } \vec{F}^{ s,y, \vec{z} } ( t,x ),
\quad s.t. \quad \vec{F}^{ s,y, \vec{z} } ( T,x ) = \big( \Phi ( s,y, x, \vec{z} ), x, {x}^{2}, \ldots, {x}^{n} \big), \\
\label{eq: extended HJB equations :eq}%%
\end{equation}
and that it satisfies the following square-integrability condition:
\begin{equation}
\mathbb{E} \bigg[ \int_{t}^{T} \| \mathscr{D}_{ \tilde{u} ( v, {X}^{ t,x, \tilde{u} }_{v} ) } \vec{F}^{ s,y, \vec{z} } ( v, {X}^{ t,x, \tilde{u} }_{v} ) \|^{2} dv \bigg] < \infty,
\quad \forall ( t,x ) \in [ 0,T ) \times \mathbb{R};
\label{eq: square-integrability condition :eq}%%
\end{equation}
Then, $\vec{F}^{ s,y, \vec{z} } ( t,x ) = \mathbb{E}_{t} [ \vec{F}^{ s,y, \vec{z} } ( T, {X}^{ t,x, \tilde{u} }_{T} ) ]$.
Furthermore, if ${U}^{ s,y, \cdot } ( t,x ) \in {C}^{1}$, $\mathcal{D}_{\zeta } {U}^{ s,y, \cdot } ( t,x ) \in {C}^{0}$ and
\begin{equation*}
\left\{ \begin{aligned}
& \mathbb{E} \bigg[ \sup_{ v \in [ t, t + \varepsilon ] } | \mathcal{D}_{\zeta } \vec{F}^{ t,x, \vec{z} } ( v, {X}^{ t,x, \zeta }_{v} ) |
                  + \int_{t}^{ t + \varepsilon } \| \mathscr{D}_{\zeta }^{1} \vec{F}^{ t, x, \vec{z} } ( v, {X}^{ t,x, \zeta }_{v} ) \|^{2} dv \bigg] < \infty, \\
& \exists \varepsilon \in ( 0, T-t ), ~ \forall \zeta \in \liminf_{ \varepsilon \downarrow 0 } \mathcal{Z}^{ t, \tilde{u} }_{\varepsilon }, ~ \forall ( t,x, \vec{z} ) \in [ 0,T ) \times \mathbb{R} \times \mathbb{R}^{n},
\end{aligned} \right.
\end{equation*}
then for all $( t,x ) \in [ 0,T ) \times \mathbb{R}$, $\zeta \in \liminf_{ \varepsilon \downarrow 0 } \mathcal{Z}^{ t, \tilde{u} }_{\varepsilon }$ and sufficiently small $\varepsilon \in ( 0, T-t ]$,
\begin{equation}
  J ( t,x, \tilde{u}^{ t, \varepsilon, \zeta } )
= J ( t,x, \tilde{u} ) + \big\langle \big( 1, \vec{\lambda }^{ t,x, \vec{m} ( t,x ) } ( t,x ) \big), \mathcal{D}_{\zeta } \vec{F}^{ t,x, \vec{m} ( t,x ) } ( t,x ) \big\rangle \varepsilon + o ( \varepsilon ).
\label{eq: first-order perturbation :eq}%%
\end{equation}
Therefore, $\tilde{u}$ is a CNEC, if for all $( t,x ) \in [ 0,T ) \times \mathbb{R}$,
\begin{equation}
\tilde{u} ( t,x ) \in \argmax_{ \zeta \in \liminf_{ \varepsilon \downarrow 0 } \mathcal{Z}^{ t, \tilde{u} }_{\varepsilon } } 
                      \big\langle \big( 1, \vec{\lambda }^{ t,x, \vec{m} ( t,x ) } ( t,x ) \big), \mathcal{D}_{\zeta } \vec{F}^{ t,x, \vec{m} ( t,x ) } ( t,x ) \big\rangle.
\label{eq: equilibrium condition :eq}%%
\end{equation}
\end{theorem}

\begin{remark}\label{rem: Lipschitz continuity in z}
If $\Phi ( t,x,y, \cdot )$ is Lipschitz on $\mathbb{R}^{n}$, for which the Lipschitz continuity parameter is uniformly bounded for all $y \in \mathbb{R}$, as assumed in \cite{Wang-Liu-Bensoussan-Yiu-Wei-2025},
then $\vec{\lambda }^{ s,y, \vec{z} } ( t,x ) = \mathbb{E}_{t} [ {\partial }_{ \vec{z} } \Phi ( s,y, {X}^{ t,x, \tilde{u} }_{T}, \vec{z} ) ]$.
Moreover, given that $\vec{\lambda }^{ s,y, \vec{z} } \in {C}^{1,2}$,
one obtains that $\mathcal{D}_{ \tilde{u} ( \tau, {X}^{ t,x, \tilde{u} }_{\tau } ) } \vec{\lambda }^{ s,y, \vec{z} } ( \tau, {X}^{ t,x, \tilde{u} }_{\tau } ) = 0$ for a.e. $\tau \in [ t,T ]$, $\mathbb{P}$-a.s.,
since $\{ \vec{\lambda }^{ s,y, \vec{z} } ( \tau, {X}^{ t,x, \tilde{u} }_{\tau } ) \}_{ \tau \in [ t,T ] }$ is an $( \mathbb{F}, \mathbb{P} )$-martingale.
Hence, $\mathcal{D}_{ \tilde{u} ( t,x ) } \vec{\lambda }^{ s,y, \vec{z} } ( t,x )$ vanishes at the point $( t,x )$, provided that $\tilde{u}$ is continuous.
In this paper, due to the absence of such Lipschitz continuity in general, 
the additional smoothness assumptions ${U}^{ s,y, \cdot } ( t,x ) \in {C}^{1}$ and $\mathcal{D}_{\zeta } {U}^{ s,y, \cdot } ( t,x ) \in {C}^{0}$ are imposed for the asymptotic analysis.
See, e.g., the proof for \Cref{lem: second-order perturbation}.
Notably, the interchange of differentiation ${\partial }_{ \vec{z} } \mathcal{D}_{\zeta } {U}^{ s,y, \vec{z} } ( t,x ) = \mathcal{D}_{\zeta } {\partial }_{ \vec{z} } {U}^{ s,y, \vec{z} } ( t,x )$ is not needed in this paper.
\end{remark}

Indeed, \Cref{thm: sufficient condition for CNEC} with the arbitrariness of $( t,x ) \in [ 0,T ] \times \mathbb{R}$ provides the sufficient condition for the CNECs satisfying \cref{eq: stronger definition for CNEC :eq}.
For \Cref{def: closed-loop Nash equilibrium control}, the regions $[ 0,T ) \times \mathbb{R}$ and $[ 0,T ] \times \mathbb{R}$ that appear everywhere in \Cref{thm: sufficient condition for CNEC} can be replaced by
$\{ ( t, {X}^{ 0, {X}_{0}, \tilde{u} }_{t} ( \omega ) ): ( t, \omega ) \in [ 0,T ] \times \Omega \}$ and $\{ ( t, {X}^{ 0, {X}_{0}, \tilde{u} }_{t} ( \omega ) ): ( t, \omega ) \in [ 0,T ) \times \Omega \}$, respectively.
In particular, under the conditions in \Cref{thm: sufficient condition for CNEC}, $\tilde{u}$ is a CNEC
only if \cref{eq: equilibrium condition :eq} holds on $\{ ( t, {X}^{ 0, {X}_{0}, \tilde{u} }_{t} ( \omega ) ): dt \times d \mathbb{P}-a.e. ~ ( t, \omega ) \in [ 0,T ) \times \Omega \}$.
For brevity, hereafter, we consider PDEs implemented on $[ 0,T ] \times \mathbb{R}$ since we only study the sufficient conditions for equilibrium strategies.
Interested readers can establish parallel arguments for $( t,x ) \in \{ ( t, {X}^{ 0, {X}_{0}, \tilde{u} }_{t} ( \omega ) ): ( t, \omega ) \in [ 0,T ] \times \Omega \}$.

\begin{example}\label{exp: linear control problem}
For the control problem formulated by \cref{eq: linear control problem :eq},
in which the parameters $( A,B,C,D,F )$ are continuous functions on $[ 0,T ]$ with $D > 0$ and $\psi ( t, {z}_{2}, \ldots, {z}_{n} )$ exhibits polynomial growth in $( {z}_{2}, \ldots, {z}_{n} ) \in \mathbb{R}^{n-1}$,
\cite{Wang-Liu-Bensoussan-Yiu-Wei-2025} showed that the $x$-independent function
\begin{equation}
\tilde{u} ( t,x ) = \frac{ {\beta }_{t} }{ {D}_{t} } {e}^{ - \int_{t}^{T} {A}_{v} dv } - \frac{ {F}_{t} }{ {D}_{t} },
\label{eq: derived CNEC :eq}
\end{equation}
is a CNEC, provided that $\beta $ is a square-integrable function that fulfills the following system on $[ 0,T )$:
\begin{equation}
\left\{ \begin{aligned}
& 0 = \frac{ {B}_{t} }{ {D}_{t} } + {\beta }_{t} \sum_{ 1 \le j \le \frac{n}{2} } \frac{ (2j)! }{ (j-1)! } \Big( \frac{1}{2} \int_{t}^{T} | {\beta }_{s} |^{2} ds \Big)^{j-1} {\psi }_{ {z}_{2j} }, \\
& 0 > \sum_{ 1 \le j \le \frac{n}{2} } \frac{ (2j)! }{ (j-1)! } \Big( \frac{1}{2} \int_{t}^{T} | {\beta }_{s} |^{2} ds \Big)^{j-1} {\psi }_{ {z}_{2j} }, \\
& {\psi }_{ {z}_{2j} } = {\psi }_{ {z}_{2j} } \bigg( t, \int_{t}^{T} | {\beta }_{s} |^{2} ds, \ldots, {1}_{\{ n \in 2 \mathbb{N} \}} \frac{ n! }{  \frac{n}{2}! } \Big( \frac{1}{2} \int_{t}^{T} | {\beta }_{s} |^{2} ds \Big)^{ \frac{n}{2} } \bigg).
\end{aligned} \right.
\label{eq: beta system :eq}%%
\end{equation}
Interested readers can refer to \cite[Remarks 5.1 and 5.3]{Wang-Liu-Bensoussan-Yiu-Wei-2025} for the economic implications of this solution.
In particular, as shown in \cite[Example 5.6]{Wang-Liu-Bensoussan-Yiu-Wei-2025},
if $\psi = \sum_{j=2}^{n} (-1)^{j+1} {\kappa }_{j} {z}_{j}$ with all ${\kappa }_{j} \ge 0$ and $\sum_{ 1 \le j \le \frac{n}{2} } {\kappa }_{2j} > 0$,
then \cref{eq: beta system :eq} is reduced to solving ${B}_{t} {\beta }_{t} \ge 0$ and the polynomial equation
\begin{equation*}
\frac{1}{2} \int_{t}^{T} \bigg| \frac{ {B}_{s} }{ {D}_{s} } \bigg|^{2} ds = \int_{0}^{y} \bigg| \sum_{ 0 \le j \le \frac{n}{2} - 1 } {\kappa }_{2j+2} \frac{ {z}^{j} }{j!} \bigg|^{2} dz
\quad for \quad y = \frac{1}{2} \int_{t}^{T} | {\beta }_{s} |^{2} ds.
\end{equation*}
\end{example}

\section{Characterization of CSES}
\label{sec: Characterization of CSES}

Suppose that $\tilde{u} \in \mathcal{U}$ is a CNEC derived by applying \Cref{thm: sufficient condition for CNEC}, e.g., the derived CNECs in \cite[Section 5.1]{Wang-Liu-Bensoussan-Yiu-Wei-2025}.
In this section, we seek sufficient conditions for $\tilde{u}$ to be a CSES.
Provided that $\tilde{u} ( t,x ) \in \liminf_{ \varepsilon \downarrow 0 } \mathcal{Z}^{ t, \tilde{u} }_{\varepsilon }$ is the unique maximum point for \cref{eq: equilibrium condition :eq},
it follows from \cref{eq: first-order perturbation :eq} that
\begin{align*}
  \liminf_{\varepsilon \downarrow 0} \frac{1}{\varepsilon } \big( J ( t,x, \tilde{u} ) - J ( t,x, \tilde{u}^{ t, \varepsilon, \zeta } ) \big)
& = - \big\langle \big( 1, \vec{\lambda }^{ t,x, \vec{m} ( t,x ) } ( t,x ) \big), \mathcal{D}_{\zeta } \vec{F}^{ t,x, \vec{m} ( t,x ) } ( t,x ) \big\rangle \\
& > - \big\langle \big( 1, \vec{\lambda }^{ t,x, \vec{m} ( t,x ) } ( t,x ) \big), \mathcal{D}_{ \tilde{u} ( t,x ) } \vec{F}^{ t,x, \vec{m} ( t,x ) } ( t,x ) \big\rangle \\
& = 0, \quad \forall \zeta \ne \tilde{u} ( t,x ).
\end{align*}
Therefore, to examine whether $\tilde{u}$ is a CSES, it suffices to show that
\begin{equation}
\left\{ \begin{aligned}
& J ( t,x, \tilde{u} ) \ge J ( t,x, \tilde{u}^{ t, \varepsilon, \tilde{u} ( t,x ) } ), \\
& \forall \varepsilon \in ( 0, {\epsilon }_{t} ), ~ \exists {\epsilon }_{t} \in ( 0, T-t ), ~ \mathbb{P}-a.s., ~ a.e. ~ t \in [ 0,T ).
\end{aligned} \right.
\label{eq: sufficient condition for CSEC :eq}%%
\end{equation}
However, since plugging \cref{eq: extended HJB equations :eq} into \cref{eq: first-order perturbation :eq} yields $J ( t,x, \tilde{u}^{ t, \varepsilon, \tilde{u} ( t,x ) } ) = J ( t,x, \tilde{u} ) + o ( \varepsilon )$,
we should investigate the higher-order expansion for $J ( t,x, \tilde{u}^{ t, \varepsilon, \tilde{u} ( t,x ) } ) - J ( t,x, \tilde{u} )$ with respect to $\varepsilon $ to establish the second-order derivative equilibrium condition.

\begin{lemma}\label{lem: higher-order expansion rule}
Suppose that the $\mathbb{R}^{n+1}$-valued function $\vec{f} \in {C}^{k,2k}$ for some positive integer $k$, and
\begin{equation*}
\left\{ \begin{aligned}
& \sum_{j=1}^{k} \mathbb{E} \bigg[ \sup_{ v \in [ t, t + \varepsilon ] } | \mathcal{D}_{\zeta }^{j} \vec{f} ( v, {X}^{ t,x, \zeta }_{v} ) |
                                 + \int_{t}^{ t + \varepsilon } \| \mathscr{D}_{\zeta }^{j} \vec{f} ( v, {X}^{ t,x, \zeta }_{v} ) \|^{2} dv \bigg] < \infty, \\
& \exists \varepsilon \in ( 0, T-t ), ~ \forall \zeta \in \liminf_{ \varepsilon \downarrow 0 } \mathcal{Z}^{ t, \tilde{u} }_{\varepsilon }, ~ \forall ( t,x ) \in [ 0,T ) \times \mathbb{R}.
\end{aligned} \right.
\end{equation*}
Then, for all $( t,x ) \in [ 0,T ) \times \mathbb{R}$, $\zeta \in \liminf_{ \varepsilon \downarrow 0 } \mathcal{Z}^{ t, \tilde{u} }_{\varepsilon }$ and sufficiently small $\varepsilon \in ( 0, T-t ]$,
\begin{equation*}
\mathbb{E}_{t} [ \vec{f} ( t + \varepsilon, {X}^{ t,x, \zeta }_{ t + \varepsilon } ) ] = \vec{f} ( t,x ) + \sum_{j=1}^{k} \frac{1}{j!} \mathcal{D}_{\zeta }^{j} \vec{f} ( t,x ) {\varepsilon }^{j} + o ( {\varepsilon }^{j} ).
\end{equation*}
\end{lemma}

\begin{proof}
For $k = 1$, after applying It\^o's rule to $\{ \vec{f} ( s, {X}^{ t,x, \zeta }_{s} ) \}_{ s \in [ t, t + \varepsilon ] }$ and taking the conditional expectation, one obtains
\begin{equation*}
| \mathbb{E}_{t} [ \vec{f} ( t + \varepsilon, {X}^{ t,x, \zeta }_{ t + \varepsilon } ) - \vec{f} ( t,x ) - \mathcal{D}_{\zeta }^{j} \vec{f} ( t,x ) \varepsilon ] |
\le \int_{t}^{ t + \varepsilon } \mathbb{E}_{t} [ | \mathcal{D}_{\zeta } \vec{f} ( s, {X}^{ t,x, \zeta }_{s} ) - \mathcal{D}_{\zeta } \vec{f} ( t,x ) | ] ds 
  = o ( \varepsilon ), 
\end{equation*}
where the last equality holds because 
$\mathbb{E}_{t} [ | \mathcal{D}_{\zeta } \vec{f} ( s, {X}^{ t,x, \zeta }_{s} ) - \mathcal{D}_{\zeta } \vec{f} ( t,x ) | ]$ is continuous in $s$ (due to the given uniform integrability)
and vanishes at $s = t$.
Similarly, for $k=2$,
\begin{align*}
& \bigg| \mathbb{E}_{t} \bigg[ \vec{f} ( t + \varepsilon, {X}^{ t,x, \zeta }_{ t + \varepsilon } ) - \vec{f} ( t,x ) 
                             - \mathcal{D}_{\zeta }^{j} \vec{f} ( t,x ) \varepsilon - \frac{1}{2} \mathcal{D}_{\zeta }^{2} \vec{f} ( t,x ) {\varepsilon }^{2} \bigg] \bigg| \\
& = \bigg| \mathbb{E}_{t} \bigg[ \int_{t}^{ t + \varepsilon } \big( \mathcal{D}_{\zeta } \vec{f} ( s, {X}^{ t,x, \zeta }_{s} ) - \mathcal{D}_{\zeta } \vec{f} ( t,x ) \big) ds 
         - \mathcal{D}_{\zeta }^{2} \vec{f} ( t,x ) \int_{t}^{ t + \varepsilon } ds \int_{t}^{s} dv \bigg] \bigg| \\
& \le \int_{t}^{ t + \varepsilon } ds \int_{t}^{s} \mathbb{E}_{t} [ | \mathcal{D}_{\zeta }^{2} \vec{f} ( v, {X}^{ t,x, \zeta }_{v} ) - \mathcal{D}_{\zeta }^{2} \vec{f} ( t,x ) | ] dv \\
& = o ( {\varepsilon }^{2} ).
\end{align*}
In the same manner, one can derive the desired result for $\vec{f} \in {C}^{k,2k}$ with an integer $k \ge 3$.
\end{proof}

\begin{lemma}\label{lem: second-order perturbation}
Suppose that $\tilde{u} \in \mathcal{U}$.
Assume that $\vec{F}^{ s,y, \vec{z} } = ( {U}^{ s,y, \vec{z} }, \vec{m} ) \in {C}^{2,4}$ satisfies the conditions in \Cref{thm: sufficient condition for CNEC} and \Cref{lem: higher-order expansion rule} with $k=2$,
and that $\mathcal{D}_{\zeta }^{k} {U}^{ s,y, \cdot } ( t,x ) \in {C}^{2-k}$ for $k = 0,1,2$.
Then, for all $( t,x ) \in [ 0,T ) \times \mathbb{R}$, $\zeta \in \liminf_{ \varepsilon \downarrow 0 } \mathcal{Z}^{ t, \tilde{u} }_{\varepsilon }$ and sufficiently small $\varepsilon \in ( 0, T-t ]$,
\begin{align}
  J ( t,x, \tilde{u}^{ t, \varepsilon, \zeta } )
& = J ( t,x, \tilde{u} ) 
  + \big\langle \big( 1, \vec{\lambda }^{ t,x, \vec{m} ( t,x ) } ( t,x ) \big), \mathcal{D}_{\zeta } \vec{F}^{ t,x, \vec{m} ( t,x ) } ( t,x ) \big\rangle \varepsilon \notag \\
& \quad + \frac{1}{2} \Big( \big\langle \big( 1, \vec{\lambda }^{ t,x, \vec{m} ( t,x ) } ( t,x ) \big), \mathcal{D}_{\zeta }^{2} \vec{F}^{ t,x, \vec{m} ( t,x ) } ( t,x ) \big\rangle 
                          + \big\langle \big( \mathcal{D}_{\zeta } \vec{m} ( t,x ) \big) {\mu }^{ t,x, \vec{m} ( t,x ) } ( t,x ), \mathcal{D}_{\zeta } \vec{m} ( t,x ) \big\rangle \notag \\
& \quad \qquad            + 2 \langle {\partial }_{ \vec{z} } \mathcal{D}_{\zeta } {U}^{ t,x, \vec{z} } ( t,x ) |_{ \vec{z} = \vec{m} ( t,x ) }, \mathcal{D}_{\zeta } \vec{m} ( t,x ) \rangle \Big) {\varepsilon }^{2}
  + o ( {\varepsilon }^{2} ).
\label{eq: second-order perturbation :eq}%%
\end{align}
\end{lemma}

\begin{proof}
By applying \Cref{lem: higher-order expansion rule} to $\vec{F}^{ s,y, \vec{z} } \in {C}^{2,4}$, one obtains
\begin{align*}
& J ( t,x, \tilde{u}^{ t, \varepsilon, \zeta } ) - J ( t,x, \tilde{u} ) \\
& = \mathbb{E}_{t} \big[ \Phi \big( t,x, {X}^{ t,x, \tilde{u}^{ t,x, \zeta } }_{T}, \vec{m}^{ \tilde{u}^{ t, \varepsilon, \zeta } } ( t,x ) \big) \big]
  - \mathbb{E}_{t} \big[ \Phi \big( t,x, {X}^{ t,x, \tilde{u} }_{T}, \vec{m} ( t,x ) \big) \big] \\
& = \mathbb{E}_{t} [ {U}^{ t,x, \mathbb{E}_{t} [ \vec{m} ( t + \varepsilon, {X}^{ t,x, \zeta }_{ t + \varepsilon } ) ] } ( t + \varepsilon, {X}^{ t,x, \zeta }_{ t + \varepsilon } ) ]
  - {U}^{ t,x, \vec{m} ( t,x ) } ( t,x ) \\
& = \mathbb{E}_{t} [ {U}^{ t,x, \mathbb{E}_{t} [ \vec{m} ( t + \varepsilon, {X}^{ t,x, \zeta }_{ t + \varepsilon } ) ] } ( t + \varepsilon, {X}^{ t,x, \zeta }_{ t + \varepsilon } ) ]
  - {U}^{ t,x, \mathbb{E}_{t} [ \vec{m} ( t + \varepsilon, {X}^{ t,x, \zeta }_{ t + \varepsilon } ) ] } ( t,x ) \\
& \quad + {U}^{ t,x, \vec{m} ( t,x ) + \mathcal{D}_{\zeta } \vec{m} ( t,x ) \varepsilon + \frac{1}{2} \mathcal{D}_{\zeta }^{2} \vec{m} ( t,x ) {\varepsilon }^{2} } ( t,x )
        - {U}^{ t,x, \vec{m} ( t,x ) } ( t,x )
        + o ( {\varepsilon }^{2} )  \\
& = \mathcal{D}_{\zeta } {U}^{ t,x, \mathbb{E}_{t} [ \vec{m} ( t + \varepsilon, {X}^{ t,x, \zeta }_{ t + \varepsilon } ) ] } ( t,x ) \varepsilon
  + \frac{1}{2} \mathcal{D}_{\zeta }^{2} {U}^{ t,x, \mathbb{E}_{t} [ \vec{m} ( t + \varepsilon, {X}^{ t,x, \zeta }_{ t + \varepsilon } ) ] } ( t,x ) {\varepsilon }^{2} \\
& \quad + \Big\langle \vec{\lambda }^{ t,x, \vec{m} ( t,x ) } ( t,x ), \mathcal{D}_{\zeta } \vec{m} ( t,x ) \varepsilon + \frac{1}{2} \mathcal{D}_{\zeta }^{2} \vec{m} ( t,x ) {\varepsilon }^{2} \Big\rangle \\
& \quad + \frac{1}{2} \big\langle \big( \mathcal{D}_{\zeta } \vec{m} ( t,x ) \big) {\mu }^{ t,x, \vec{m} ( t,x ) } ( t,x ), \mathcal{D}_{\zeta } \vec{m} ( t,x ) \big\rangle {\varepsilon }^{2}
        + o ( {\varepsilon }^{2} ),
\end{align*}
where
\begin{align*}
& \mathcal{D}_{\zeta } {U}^{ t,x, \mathbb{E}_{t} [ \vec{m} ( t + \varepsilon, {X}^{ t,x, \zeta }_{ t + \varepsilon } ) ] } ( t,x ) 
= \mathcal{D}_{\zeta } {U}^{ t,x, \vec{m} ( t,x ) } ( t,x ) 
+ \langle {\partial }_{ \vec{z} } \mathcal{D}_{\zeta } {U}^{ t,x, \vec{z} } ( t,x ) |_{ \vec{z} = \vec{m} ( t,x ) }, \mathcal{D}_{\zeta } \vec{m} ( t,x ) \rangle \varepsilon + o ( \varepsilon ), \\
& \mathcal{D}_{\zeta }^{2} {U}^{ t,x, \mathbb{E}_{t} [ \vec{m} ( t + \varepsilon, {X}^{ t,x, \zeta }_{ t + \varepsilon } ) ] } ( t,x ) 
= \mathcal{D}_{\zeta }^{2} {U}^{ t,x, \vec{m} ( t,x ) } ( t,x ) + o (1).
\end{align*}
By rearranging the terms according to the degree of $\varepsilon $, one immediately obtains \cref{eq: second-order perturbation :eq}.
\end{proof}

\Cref{lem: second-order perturbation} implies that a sufficient condition for a CSES is \cref{eq: equilibrium condition :eq} with
\begin{align*}
0 & > \big\langle \big( 1, \vec{\lambda }^{ t,x, \vec{m} ( t,x ) } ( t,x ) \big), 
                  \mathcal{D}_{\zeta }^{2} \vec{F}^{ t,x, \vec{m} ( t,x ) } ( t,x ) \big\rangle
    + \big\langle \big( \mathcal{D}_{\zeta } \vec{m} ( t,x ) \big) {\mu }^{ t,x, \vec{m} ( t,x ) } ( t,x ), \mathcal{D}_{\zeta } \vec{m} ( t,x ) \big\rangle \\
  & \quad + 2 \langle {\partial }_{ \vec{z} } \mathcal{D}_{\zeta } {U}^{ t,x, \vec{z} } ( t,x ) |_{ \vec{z} = \vec{m} ( t,x ) }, \mathcal{D}_{\zeta } \vec{m} ( t,x ) \rangle
\end{align*}
for all $\zeta \in \liminf_{ \varepsilon \downarrow 0 } \mathcal{Z}^{ t, \tilde{u} }_{\varepsilon }$ and $( t,x ) \in [ 0,T ) \times \mathbb{R}$.
However, in what follows, we show that the abovementioned strict inequality does not necessarily hold in the case where $\tilde{u}_{x} ( t, \cdot ) \equiv 0$; see the following \Cref{thm: second-order derivative condition for CSES}.
Thus, the abovementioned condition cannot be used to examine whether the CNEC derived in \Cref{exp: linear control problem} is a CSES.

\begin{theorem}\label{thm: second-order derivative condition for CSES}
Suppose that $\tilde{u} \in \mathcal{U} \cap {C}^{1,2}$.
Assume that the following properties hold:
\begin{itemize}
\item $\vec{F}^{ s,y, \vec{z} } = ( {U}^{ s,y, \vec{z} }, \vec{m} ) \in {C}^{2,4}$ satisfies the conditions in \Cref{thm: sufficient condition for CNEC} and \Cref{lem: higher-order expansion rule} with $k=2$.
\item $\mathcal{D}_{\zeta }^{k} {U}^{ s,y, \cdot } ( t,x ) \in {C}^{2-k}$ for $k = 0,1,2$.
\item For all $( t,x ) \in [ 0,T ) \times \mathbb{R}$, $\tilde{u} ( t,x ) \in \liminf_{ \varepsilon \downarrow 0 } \mathcal{Z}^{ t, \tilde{u} }_{\varepsilon }$
      fulfills the following first-order and second-order derivative optimality conditions for \cref{eq: equilibrium condition :eq}:
\begin{align}
& 0 = \big\langle \big( 1, \vec{\lambda }^{ t,x, \vec{m} ( t,x ) } ( t,x ) \big), \vec{H}^{ t,x, \vec{m} ( t,x ) }_{u} \big( t,x, \tilde{u} ( t,x ) \big) \big\rangle, 
\label{eq: first-order derivative condition :eq}%%
\\
& 0 > \big\langle \big( 1, \vec{\lambda }^{ t,x, \vec{m} ( t,x ) } ( t,x ) \big), \vec{H}^{ t,x, \vec{m} ( t,x ) }_{uu} ( t,x, \zeta ) \big\rangle, 
  \quad \forall \zeta \in \liminf_{ \varepsilon \downarrow 0 } \mathcal{Z}^{ t, \tilde{u} }_{\varepsilon }.
\label{eq: second-order derivative condition :eq}%%
\end{align}
\end{itemize}
Then,
\begin{align*}
& J ( t,x, \tilde{u}^{ t, \varepsilon, \tilde{u} ( t,x ) } ) - J ( t,x, \tilde{u} ) \\ 
& = - \frac{1}{2} {\varepsilon }^{2} \hat{\sigma }^{ \tilde{u} ( t,x ) } ( t,x ) | \tilde{u}_{x} ( t,x ) |^{2}
          \big\langle \big( 1, \vec{\lambda }^{ t,x, \vec{m} ( t,x ) } ( t,x ) \big), \vec{H}^{ t,x, \vec{m} ( t,x ) }_{uu} \big( t,x, \tilde{u} ( t,x ) \big) \big\rangle \\
& \quad - {\varepsilon }^{2} \hat{\sigma }^{ \tilde{u} ( t,x ) } ( t,x ) \tilde{u}_{x} ( t,x ) 
          \big\langle \big( 1, \vec{\lambda }^{ t,x, \vec{m} ( t,x ) } ( t,x ) \big), \vec{H}^{ t,x, \vec{m} ( t,x ) }_{xu} \big( t,x, \tilde{u} ( t,x ) \big) \big\rangle
        + o ( {\varepsilon }^{2} ) \\
& = \frac{1}{2} {\varepsilon }^{2} \hat{\sigma }^{ \tilde{u} ( t,x ) } ( t,x ) | \tilde{u}_{x} ( t,x ) |^{2}
    \big\langle \big( 1, \vec{\lambda }^{ t,x, \vec{m} ( t,x ) } ( t,x ) \big), \vec{H}^{ t,x, \vec{m} ( t,x ) }_{uu} \big( t,x, \tilde{u} ( t,x ) \big) \big\rangle \\
& \quad - {\varepsilon }^{2} \hat{\sigma }^{ \tilde{u} ( t,x ) } ( t,x ) \tilde{u}_{x} ( t,x ) 
          \big\langle \big( 1, \vec{\lambda }^{ t,x, \vec{m} ( t,x ) } ( t,x ) \big), 
                      {\partial }_{x} \vec{H}^{ t,y, \vec{z} }_{u} \big( t,x, \tilde{u} ( t,x ) \big) \big\rangle \big|_{ ( y, \vec{z} ) = ( x, \vec{m} ( t,x ) ) } 
        + o ( {\varepsilon }^{2} ).
\end{align*}
Therefore, $\tilde{u}$ is a CSES, if
\begin{align}
0 & < \tilde{u}_{x} ( t,x ) \big\langle \big( 1, \vec{\lambda }^{ t,x, \vec{m} ( t,x ) } ( t,x ) \big), \vec{H}^{ t,x, \vec{m} ( t,x ) }_{xu} \big( t,x, \tilde{u} ( t,x ) \big) \big\rangle \notag \\
  & \quad + \frac{1}{2} | \tilde{u}_{x} ( t,x ) |^{2} \big\langle \big( 1, \vec{\lambda }^{ t,x, \vec{m} ( t,x ) } ( t,x ) \big), \vec{H}^{ t,x, \vec{m} ( t,x ) }_{uu} \big( t,x, \tilde{u} ( t,x ) \big) \big\rangle,
\label{eq: equilibrium condition second-order I :eq}%%
\end{align}
or
\begin{equation}
\left\{ \begin{aligned}
0 & \ne \tilde{u}_{x} ( t,x ), \\
0 & \le \tilde{u}_{x} ( t,x ) \big\langle \big( 1, \vec{\lambda }^{ t,x, \vec{m} ( t,x ) } ( t,x ) \big), 
                                          {\partial }_{x} \vec{H}^{ t,y, \vec{z} }_{u} \big( t,x, \tilde{u} ( t,x ) \big) \big\rangle \big|_{ ( y, \vec{z} ) = ( x, \vec{m} ( t,x ) ) }.
\end{aligned} \right.
\label{eq: equilibrium condition second-order II :eq}%%
\end{equation}
\end{theorem}

\begin{proof}
Clearly, the derivative optimality conditions \cref{eq: first-order derivative condition :eq} and \cref{eq: second-order derivative condition :eq} imply that 
$\tilde{u} ( t,x ) \in \liminf_{ \varepsilon \downarrow 0 } \mathcal{Z}^{ t, \tilde{u} }_{\varepsilon }$ is the unique maximum point for \cref{eq: equilibrium condition :eq}.
Consequently, to determine whether $\tilde{u}$ is a CSES, it suffices to show \cref{eq: sufficient condition for CSEC :eq}.
By plugging \cref{eq: extended HJB equations :eq} into \cref{eq: second-order perturbation :eq}, one can proceed with
\begin{equation}
J ( t,x, \tilde{u}^{ t, \varepsilon, \tilde{u} ( t,x ) } ) - J ( t,x, \tilde{u} )
= \frac{1}{2} {\varepsilon }^{2} \big\langle \big( 1, \vec{\lambda }^{ t,x, \vec{m} ( t,x ) } ( t,x ) \big), \mathcal{A}_{\zeta } \mathcal{D}_{\zeta }^{2} \vec{F}^{ t,x, \vec{m} ( t,x ) } ( t,x ) \big\rangle
+ o ( {\varepsilon }^{2} ).
\label{eq: second-order expansion for given perturbation :eq}%%
\end{equation}
Notably, the ${\varepsilon }^{2}$-term on the right-hand side of \cref{eq: second-order expansion for given perturbation :eq} can be further simplified 
by exploiting the PDE \cref{eq: extended HJB equations :eq} and the first-order derivative condition \cref{eq: first-order derivative condition :eq}.
Indeed, 
\begin{align*}
  \mathcal{D}_{\zeta }^{2} \vec{F}^{ s,y, \vec{z} }
& = \mathcal{D}_{\zeta } ( \mathcal{D}_{\zeta } \vec{F}^{ s,y, \vec{z} } - \mathcal{A}_{\xi } \mathcal{D}_{\xi } \vec{F}^{ s,y, \vec{z} } ) \\
& = \mathcal{D}_{\zeta } \mathcal{A}_{\xi } \big( \vec{H}^{ s,y, \vec{z} } ( t,x, \zeta ) - \vec{H}^{ s,y, \vec{z} } ( t,x, \xi ) \big) \\
& = \mathcal{A}_{\xi } \mathcal{D}_{\zeta } \big( \vec{H}^{ s,y, \vec{z} } ( t,x, \zeta ) - \vec{H}^{ s,y, \vec{z} } ( t,x, \xi ) \big) \\
& \quad - \big( ( \mathcal{D}_{\zeta } \tilde{u} ) \mathcal{A}_{\xi } {\partial }_{\xi } 
              + 2 \tilde{u}_{x} \hat{\sigma }^{\zeta } {\partial }_{x} \mathcal{A}_{\xi } {\partial }_{\xi } 
              - | \tilde{u}_{x} |^{2} \hat{\sigma }^{\zeta } \mathcal{A}_{\xi } {\partial }_{\xi }^{2} \big) \vec{H}^{ s,y, \vec{z} } ( t,x, \xi ),
\end{align*}
where the arguments $( t,x )$ suppressed for notational simplicity.
Consequently,
\begin{align*}
  \mathcal{A}_{\zeta } \mathcal{D}_{\zeta }^{2} \vec{F}^{ s,y, \vec{z} } ( t,x )
& = - \vec{H}^{ s,y, \vec{z} }_{u} \big( t,x, \tilde{u} ( t,x ) \big) \mathcal{D}_{\zeta } \tilde{u} ( t,x ) 
    - 2 \vec{H}^{ s,y, \vec{z} }_{xu} \big( t,x, \tilde{u} ( t,x ) \big) \tilde{u}_{x} ( t,x ) \hat{\sigma } \big( t,x, \tilde{u} ( t,x ) \big) \\
& \quad - \vec{H}^{ s,y, \vec{z} }_{uu} \big( t,x, \tilde{u} ( t,x ) \big) | \tilde{u}_{x} ( t,x ) |^{2} \hat{\sigma } \big( t,x, \tilde{u} ( t,x ) \big) \\
& = - \vec{H}^{ s,y, \vec{z} }_{u} \big( t,x, \tilde{u} ( t,x ) \big) \mathcal{D}_{\zeta } \tilde{u} ( t,x ) 
    - 2 \hat{\sigma } \big( t,x, \tilde{u} ( t,x ) \big) \tilde{u}_{x} ( t,x ) \frac{\partial }{\partial x} \vec{H}^{ s,y, \vec{z} }_{u} \big( t,x, \tilde{u} ( t,x ) \big) \\
& \quad + \hat{\sigma } \big( t,x, \tilde{u} ( t,x ) \big) | \tilde{u}_{x} ( t,x ) |^{2} \vec{H}^{ s,y, \vec{z} }_{uu} \big( t,x, \tilde{u} ( t,x ) \big).
\end{align*}
Substituting these expressions into \cref{eq: second-order expansion for given perturbation :eq} and employing the first-order derivative optimality condition \cref{eq: first-order derivative condition :eq}
yields the desired expansion for $J ( t,x, \tilde{u}^{ t, \varepsilon, \tilde{u} ( t,x ) } ) - J ( t,x, \tilde{u} )$.
The sufficient condition \cref{eq: equilibrium condition second-order I :eq} immediately follows.
Since the second-order derivative optimality condition \cref{eq: second-order derivative condition :eq} holds,
\cref{eq: equilibrium condition second-order II :eq} also leads to \cref{eq: sufficient condition for CSEC :eq}, and hence is sufficient for a CSES.
\end{proof}

According to \Cref{thm: second-order derivative condition for CSES}, in the case where $\tilde{u}_{x} ( t,x ) = 0$, e.g., for \Cref{exp: linear control problem}, one obtains
\begin{equation*}
J ( t,x, \tilde{u}^{ t, \varepsilon, \tilde{u} ( t,x ) } ) - J ( t,x, \tilde{u} ) = o ( {\varepsilon }^{2} ).
\end{equation*}
This expression does not prevent the occurrence of $J ( t,x, \tilde{u}^{ t, \varepsilon, \zeta } ) - J ( t,x, \tilde{u} ) = K ( t,x, \varepsilon ) {\varepsilon }^{3} > 0$ for some bounded positive function $K$.
Therefore, we need to investigate the third-order expansion and establish the third-order derivative equilibrium condition for this case.

\begin{theorem}\label{thm: third-order derivative condition for CSES}
Suppose that $\tilde{u} \in \mathcal{U} \cap {C}^{2,4}$ with $\tilde{u}_{x} ( t,x ) = 0$ at the given point $( t,x ) \in [ 0,T ) \times \mathbb{R}$.
Assume that $\vec{F}^{ s,y, \vec{z} } = ( {U}^{ s,y, \vec{z} }, \vec{m} ) \in {C}^{3,6}$ 
satisfies the conditions in \Cref{thm: sufficient condition for CNEC,thm: second-order derivative condition for CSES} and \Cref{lem: higher-order expansion rule} with $k=3$,
and that $\mathcal{D}_{\zeta }^{k} {U}^{ s,y, \cdot } ( t,x ) \in {C}^{3-k}$ for $k = 0,1,2,3$.
Then,
\begin{align*}
& J ( t,x, \tilde{u}^{ t, \varepsilon, \tilde{u} ( t,x ) } ) - J ( t,x, \tilde{u} ) \\ 
& = \frac{1}{6} {\varepsilon }^{3} \big( | \mathcal{D}_{ \tilde{u} ( t,x ) } \tilde{u} ( t,x ) |^{2} + 2 | \tilde{u}_{xx} ( t,x ) \hat{\sigma }^{ \tilde{u} ( t,x ) } ( t,x ) |^{2} \big) 
    \big\langle \big( 1, \vec{\lambda }^{ t,x, \vec{m} ( t,x ) } ( t,x ) \big), \vec{H}^{ t,x, \vec{m} ( t,x ) }_{uu} \big( t,x, \tilde{u} ( t,x ) \big) \big\rangle \\
& \quad - \frac{1}{3} {\varepsilon }^{3} ( \mathcal{D}_{ \tilde{u} ( t,x ) } \tilde{u} ( t,x ) ) 
          \big\langle \big( 1, \vec{\lambda }^{ t,x, \vec{m} ( t,x ) } ( t,x ) \big), \mathcal{D}_{ \tilde{u} ( t,x ) } \vec{H}^{ t,x, \vec{m} ( t,x ) }_{u} \big( t,x, \tilde{u} ( t,x ) \big) \big\rangle \\
& \quad - \frac{1}{3} {\varepsilon }^{3} \big( 2 \mathcal{D}_{ \tilde{u} ( t,x ) } \tilde{u}_{x} ( t,x ) + 3 \tilde{u}_{xx} ( t,x ) \hat{\sigma }^{ \tilde{u} ( t,x ) }_{x} ( t,x ) \big) \hat{\sigma }^{ \tilde{u} ( t,x ) } ( t,x ) \\
& \qquad  \times \big\langle \big( 1, \vec{\lambda }^{ t,x, \vec{m} ( t,x ) } ( t,x ) \big), {\partial }_{x} \vec{H}^{ t,y, \vec{z} }_{u} \big( t,x, \tilde{u} ( t,x ) \big) \big\rangle \big|_{ ( y, \vec{z} ) = ( x, \vec{m} ( t,x ) ) } \\  
& \quad - \frac{2}{3} {\varepsilon }^{3} \tilde{u}_{xx} ( t,x ) | \hat{\sigma }^{ \tilde{u} ( t,x ) } ( t,x ) |^{2} 
          \big\langle \big( 1, \vec{\lambda }^{ t,x, \vec{m} ( t,x ) } ( t,x ) \big), {\partial }_{x}^{2} \vec{H}^{ t,y, \vec{z} }_{u} \big( t,x, \tilde{u} ( t,x ) \big) \big\rangle \big|_{ ( y, \vec{z} ) = ( x, \vec{m} ( t,x ) ) }
   + o ( {\varepsilon }^{3} ).
\end{align*}
Therefore, $\tilde{u}$ is a CSES, if
\begin{align}
0 & > \frac{1}{2} \big( | \mathcal{D}_{ \tilde{u} ( t,x ) } \tilde{u} ( t,x ) |^{2} + 2 | \tilde{u}_{xx} ( t,x ) \hat{\sigma }^{ \tilde{u} ( t,x ) } ( t,x ) |^{2} \big) 
      \big\langle \big( 1, \vec{\lambda }^{ t,x, \vec{m} ( t,x ) } ( t,x ) \big), \vec{H}^{ t,x, \vec{m} ( t,x ) }_{uu} \big( t,x, \tilde{u} ( t,x ) \big) \big\rangle \notag \\
  & \quad - ( \mathcal{D}_{ \tilde{u} ( t,x ) } \tilde{u} ( t,x ) ) 
            \big\langle \big( 1, \vec{\lambda }^{ t,x, \vec{m} ( t,x ) } ( t,x ) \big), \mathcal{D}_{ \tilde{u} ( t,x ) } \vec{H}^{ t,x, \vec{m} ( t,x ) }_{u} \big( t,x, \tilde{u} ( t,x ) \big) \big\rangle \notag \\
  & \quad - \big( 2 \mathcal{D}_{ \tilde{u} ( t,x ) } \tilde{u}_{x} ( t,x ) + 3 \tilde{u}_{xx} ( t,x ) \hat{\sigma }^{ \tilde{u} ( t,x ) }_{x} ( t,x ) \big) \hat{\sigma }^{ \tilde{u} ( t,x ) } ( t,x ) \notag \\
  & \qquad  \times \big\langle \big( 1, \vec{\lambda }^{ t,x, \vec{m} ( t,x ) } ( t,x ) \big), {\partial }_{x} \vec{H}^{ t,y, \vec{z} }_{u} \big( t,x, \tilde{u} ( t,x ) \big) \big\rangle \big|_{ ( y, \vec{z} ) = ( x, \vec{m} ( t,x ) ) } \notag \\  
  & \quad - 2 \tilde{u}_{xx} ( t,x ) | \hat{\sigma }^{ \tilde{u} ( t,x ) } ( t,x ) |^{2} 
              \big\langle \big( 1, \vec{\lambda }^{ t,x, \vec{m} ( t,x ) } ( t,x ) \big), {\partial }_{x}^{2} \vec{H}^{ t,y, \vec{z} }_{u} \big( t,x, \tilde{u} ( t,x ) \big) \big\rangle \big|_{ ( y, \vec{z} ) = ( x, \vec{m} ( t,x ) ) }.
\label{eq: equilibrium condition third-order I :eq}%%
\end{align}
or
\begin{equation}
\left\{ \begin{aligned}
0 & \ne | \mathcal{D}_{ \tilde{u} ( t,x ) } \tilde{u} ( t,x ) |^{2} + 2 | \tilde{u}_{xx} ( t,x ) \hat{\sigma }^{ \tilde{u} ( t,x ) } ( t,x ) |^{2}, \\
0 & \le ( \mathcal{D}_{ \tilde{u} ( t,x ) } \tilde{u} ( t,x ) ) 
        \big\langle \big( 1, \vec{\lambda }^{ t,x, \vec{m} ( t,x ) } ( t,x ) \big), \mathcal{D}_{ \tilde{u} ( t,x ) } \vec{H}^{ t,x, \vec{m} ( t,x ) }_{u} \big( t,x, \tilde{u} ( t,x ) \big) \big\rangle \\
  & \quad + \big( 2 \mathcal{D}_{ \tilde{u} ( t,x ) } \tilde{u}_{x} ( t,x ) + 3 \tilde{u}_{xx} ( t,x ) \hat{\sigma }^{ \tilde{u} ( t,x ) }_{x} ( t,x ) \big) \hat{\sigma }^{ \tilde{u} ( t,x ) } ( t,x ) \\
  & \qquad  \times \big\langle \big( 1, \vec{\lambda }^{ t,x, \vec{m} ( t,x ) } ( t,x ) \big), {\partial }_{x} \vec{H}^{ t,y, \vec{z} }_{u} \big( t,x, \tilde{u} ( t,x ) \big) \big\rangle \big|_{ ( y, \vec{z} ) = ( x, \vec{m} ( t,x ) ) } \\  
  & \quad + 2 \tilde{u}_{xx} ( t,x ) | \hat{\sigma }^{ \tilde{u} ( t,x ) } ( t,x ) |^{2} 
              \big\langle \big( 1, \vec{\lambda }^{ t,x, \vec{m} ( t,x ) } ( t,x ) \big), {\partial }_{x}^{2} \vec{H}^{ t,y, \vec{z} }_{u} \big( t,x, \tilde{u} ( t,x ) \big) \Big\rangle \big|_{ ( y, \vec{z} ) = ( x, \vec{m} ( t,x ) ) }.
\end{aligned} \right.
\label{eq: equilibrium condition third-order II :eq}%%
\end{equation}
\end{theorem}

\begin{proof}
Similar to the proof for \Cref{lem: second-order perturbation}, let us proceed with
\begin{align*}
& J ( t,x, \tilde{u}^{ t, \varepsilon, \tilde{u} ( t,x ) } ) - J ( t,x, \tilde{u} ) \\
& = \mathcal{A}_{\zeta } \Big( \mathbb{E}_{t} [ {U}^{ t,x, \mathbb{E}_{t} [ \vec{m} ( t + \varepsilon, {X}^{ t,x, \zeta }_{ t + \varepsilon } ) ] } ( t + \varepsilon, {X}^{ t,x, \zeta }_{ t + \varepsilon } ) ]
                             - {U}^{ t,x, \mathbb{E}_{t} [ \vec{m} ( t + \varepsilon, {X}^{ t,x, \zeta }_{ t + \varepsilon } ) ] } ( t,x ) \Big) \\
& \quad + \mathcal{A}_{\zeta } \Big( {U}^{ t,x, \vec{m} ( t,x ) + \mathcal{D}_{\zeta } \vec{m} ( t,x ) \varepsilon + \frac{1}{2} \mathcal{D}_{\zeta }^{2} \vec{m} ( t,x ) {\varepsilon }^{2} 
                                                                + \frac{1}{6} \mathcal{D}_{\zeta }^{3} \vec{m} ( t,x ) {\varepsilon }^{3} } ( t,x )
                                   - {U}^{ t,x, \vec{m} ( t,x ) } ( t,x ) \Big)
        + o ( {\varepsilon }^{3} )  \\
& = \mathcal{A}_{\zeta } \mathcal{D}_{\zeta } {U}^{ t,x, \mathbb{E}_{t} [ \vec{m} ( t + \varepsilon, {X}^{ t,x, \zeta }_{ t + \varepsilon } ) ] } ( t,x ) \varepsilon
  + \frac{1}{2} \mathcal{A}_{\zeta } \mathcal{D}_{\zeta }^{2} {U}^{ t,x, \vec{m} ( t,x ) } ( t,x ) {\varepsilon }^{2} 
  + \frac{1}{6} \mathcal{A}_{\zeta } \mathcal{D}_{\zeta }^{3} {U}^{ t,x, \vec{m} ( t,x ) } ( t,x ) {\varepsilon }^{3} \\
& \quad + \Big\langle \vec{\lambda }^{ t,x, \vec{m} ( t,x ) } ( t,x ), \frac{1}{2} \mathcal{A}_{\zeta } \mathcal{D}_{\zeta }^{2} \vec{m} ( t,x ) {\varepsilon }^{2} 
                                                                     + \frac{1}{6} \mathcal{A}_{\zeta } \mathcal{D}_{\zeta }^{3} \vec{m} ( t,x ) {\varepsilon }^{3} \Big\rangle 
        + o ( {\varepsilon }^{3} ) \\
& = \frac{1}{6} \big\langle \big( 1, \vec{\lambda }^{ t,x, \vec{m} ( t,x ) } ( t,x ) \big), \mathcal{A}_{\zeta } \mathcal{D}_{\zeta }^{3} \vec{F}^{ t,x, \vec{m} ( t,x ) } ( t,x ) \big\rangle {\varepsilon }^{3}
  + o ( {\varepsilon }^{3} ),
\end{align*}
where the last equality holds because
\begin{itemize}
\item $\mathcal{A}_{\zeta } \mathcal{D}_{\zeta } \vec{\lambda }^{ t,x, \vec{m} ( t,x ) } ( t,x ) = 0$ according to \cref{eq: extended HJB equations :eq} for all $\vec{z} \in \mathbb{R}^{n}$,
\item and $\langle ( 1, \vec{\lambda }^{ t,x, \vec{m} ( t,x ) } ( t,x ) ), \mathcal{A}_{\zeta } \mathcal{D}_{\zeta }^{2} \vec{F}^{ t,x, \vec{m} ( t,x ) } ( t,x ) \rangle = 0$
      according to \Cref{thm: second-order derivative condition for CSES} with $\tilde{u}_{x} ( t,x ) = 0$.
\end{itemize}
Similar to the proof for \Cref{thm: second-order derivative condition for CSES}, given that $\mathcal{A}_{\xi } \mathcal{D}_{\xi } \vec{F}^{ s,y, \vec{z} } \equiv 0$ and $\tilde{u}_{x} = 0$ at $( t,x )$, we have
\begin{align*}
  \mathcal{D}_{\zeta }^{3} \vec{F}^{ s,y, \vec{z} }
& = \mathcal{D}_{\zeta }^{2} ( \mathcal{D}_{\zeta } \vec{F}^{ s,y, \vec{z} } - \mathcal{A}_{\xi } \mathcal{D}_{\xi } \vec{F}^{ s,y, \vec{z} } ) \\
& = \mathcal{D}_{\zeta }^{2} \mathcal{A}_{\xi } \big( \vec{H}^{ s,y, \vec{z} } ( t,x, \zeta ) - \vec{H}^{ s,y, \vec{z} } ( t,x, \xi ) \big) \\
& = \mathcal{D}_{\zeta } \big( \mathcal{A}_{\xi } \mathcal{D}_{\zeta } 
                             + \mathcal{D}_{\zeta } ( \tilde{u} \mathcal{A}_{\xi } {\partial }_{\xi } )
                             - \tilde{u} \mathcal{D}_{\zeta } \mathcal{A}_{\xi } {\partial }_{\xi }
                             - | \tilde{u}_{x} |^{2} \hat{\sigma }^{\zeta } \mathcal{A}_{\xi } {\partial }_{\xi }^{2} \big)
    \big( \vec{H}^{ s,y, \vec{z} } ( t,x, \zeta ) - \vec{H}^{ s,y, \vec{z} } ( t,x, \xi ) \big) \\
& = \mathcal{A}_{\xi } \mathcal{D}_{\zeta }^{2} \big( \vec{H}^{ s,y, \vec{z} } ( t,x, \zeta ) - \vec{H}^{ s,y, \vec{z} } ( t,x, \xi ) \big)
  - ( \mathcal{D}_{\zeta } \tilde{u} ) \mathcal{A}_{\xi } \mathcal{D}_{\zeta } \vec{H}^{ s,y, \vec{z} }_{u} ( t,x, \xi ) \\
& \quad - \mathcal{D}_{\zeta }^{2} \big( \tilde{u} \mathcal{A}_{\xi } \vec{H}^{ s,y, \vec{z} }_{u} ( t,x, \xi ) \big)
        + \mathcal{D}_{\zeta } \big( \tilde{u} \mathcal{D}_{\zeta } \mathcal{A}_{\xi } \vec{H}^{ s,y, \vec{z} }_{u} ( t,x, \xi ) \big)
        + \mathcal{D}_{\zeta } \big( | \tilde{u}_{x} |^{2} \hat{\sigma }^{\zeta } \mathcal{A}_{\xi } \vec{H}^{ s,y, \vec{z} }_{uu} ( t,x, \xi ) \big) \\
& = \mathcal{A}_{\xi } \mathcal{D}_{\zeta }^{2} \big( \vec{H}^{ s,y, \vec{z} } ( t,x, \zeta ) - \vec{H}^{ s,y, \vec{z} } ( t,x, \xi ) \big)
  + ( \mathcal{D}_{\zeta } \tilde{u} ) ( \mathcal{D}_{\zeta } \mathcal{A}_{\xi } - \mathcal{A}_{\xi } \mathcal{D}_{\zeta } ) \vec{H}^{ s,y, \vec{z} }_{u} ( t,x, \xi ) \\
& \quad - \mathcal{D}_{\zeta }^{2} \big( \tilde{u} \mathcal{A}_{\xi } \vec{H}^{ s,y, \vec{z} }_{u} ( t,x, \xi ) \big)
        + \tilde{u} \mathcal{D}_{\zeta }^{2} \mathcal{A}_{\xi } \vec{H}^{ s,y, \vec{z} }_{u} ( t,x, \xi ) 
        + ( \mathcal{D}_{\zeta } | \tilde{u}_{x} |^{2} ) \hat{\sigma }^{\zeta } \mathcal{A}_{\xi } \vec{H}^{ s,y, \vec{z} }_{uu} ( t,x, \xi ) \\
& = \mathcal{A}_{\xi } \mathcal{D}_{\zeta }^{2} \big( \vec{H}^{ s,y, \vec{z} } ( t,x, \zeta ) - \vec{H}^{ s,y, \vec{z} } ( t,x, \xi ) \big)
  + ( | \mathcal{D}_{\zeta } \tilde{u} |^{2} + 2 | \tilde{u}_{xx} |^{2} | \hat{\sigma }^{\zeta } |^{2} ) \mathcal{A}_{\xi } \vec{H}^{ s,y, \vec{z} }_{uu} ( t,x, \xi ) \\
& \quad - \mathcal{D}_{\zeta }^{2} \big( \tilde{u} \mathcal{A}_{\xi } \vec{H}^{ s,y, \vec{z} }_{u} ( t,x, \xi ) \big)
        + \tilde{u} \mathcal{D}_{\zeta }^{2} \mathcal{A}_{\xi } \vec{H}^{ s,y, \vec{z} }_{u} ( t,x, \xi ).
\end{align*}
For $f = \mathcal{A}_{\xi } \vec{H}^{ s,y, \vec{z} }_{u} ( t,x, \xi )$, we have
\begin{align*}
  \mathcal{D}_{\zeta }^{2} ( \tilde{u} f )
& = \mathcal{D}_{\zeta } ( \tilde{u} \mathcal{D}_{\zeta } f + f \mathcal{D}_{\zeta } u + \tilde{u}_{x} \hat{\sigma }^{\zeta } {\partial }_{x} f ) \\
& = \tilde{u} ( \mathcal{D}_{\zeta }^{2} f ) + 2 ( \mathcal{D}_{\zeta } \tilde{u} ) ( \mathcal{D}_{\zeta } f ) + ( \mathcal{D}_{\zeta }^{2} \tilde{u} ) f
+ ( 4 \mathcal{D}_{\zeta } {\partial }_{x} \tilde{u} + 6 \tilde{u}_{xx} \hat{\sigma }_{x} ) \hat{\sigma } {\partial }_{x} f
+ 4 \tilde{u}_{xx} | \hat{\sigma } |^{2} {\partial }_{x}^{2} f.
\end{align*}
Consequently,
\begin{align*}
  \mathcal{D}_{\zeta }^{3} \vec{F}^{ s,y, \vec{z} }
& = \mathcal{A}_{\xi } \mathcal{D}_{\zeta }^{2} \big( \vec{H}^{ s,y, \vec{z} } ( t,x, \zeta ) - \vec{H}^{ s,y, \vec{z} } ( t,x, \xi ) \big)
  + ( | \mathcal{D}_{\zeta } \tilde{u} |^{2} + 2 | \tilde{u}_{xx} |^{2} | \hat{\sigma }^{\zeta } |^{2} ) \mathcal{A}_{\xi } \vec{H}^{ s,y, \vec{z} }_{uu} ( t,x, \xi ) \\
& \quad - \big( ( \mathcal{D}_{\zeta }^{2} \tilde{u} ) + 2 ( \mathcal{D}_{\zeta } \tilde{u} ) \mathcal{D}_{\zeta }
              + ( 4 \mathcal{D}_{\zeta } \tilde{u}_{x} + 6 \tilde{u}_{xx} \hat{\sigma }^{\zeta }_{x} ) \hat{\sigma }^{\zeta } {\partial }_{x}
              + 4 \tilde{u}_{xx} | \hat{\sigma }^{\zeta } |^{2} {\partial }_{x}^{2} \big) \mathcal{A}_{\xi } \vec{H}^{ s,y, \vec{z} }_{u} ( t,x, \xi ).
\end{align*}
Since $\mathcal{A}_{\zeta } \mathcal{A}_{\xi } \mathcal{D}_{\zeta }^{2} ( \vec{H}^{ s,y, \vec{z} } ( t,x, \zeta ) - \vec{H}^{ s,y, \vec{z} } ( t,x, \xi ) ) = 0$,
and the first-order derivative condition \cref{eq: first-order derivative condition :eq} gives
\begin{equation*}
\big\langle \big( 1, \vec{\lambda }^{ t,x, \vec{m} ( t,x ) } ( t,x ) \big), \mathcal{A}_{\zeta } ( \mathcal{D}_{\zeta }^{2} \tilde{u} ) \mathcal{A}_{\xi } \vec{H}^{ t,x, \vec{m} ( t,x ) }_{u} ( t,x, \xi ) \big\rangle = 0,
\end{equation*}
one can obtain the desired third-order expansion for $J ( t,x, \tilde{u}^{ t, \varepsilon, \tilde{u} ( t,x ) } ) - J ( t,x, \tilde{u} )$ and the sufficiency of \cref{eq: equilibrium condition second-order I :eq} for a CESE.
Furthermore, given the second-order derivative condition \cref{eq: second-order derivative condition :eq},
the sufficiency of \cref{eq: equilibrium condition second-order II :eq} for a CSES immediately arises.
\end{proof}

\begin{remark}\label{rem: state-independent case}
In particular, if $\tilde{u}_{x} \equiv 0$ on the neighborhood of $( t,x )$, e.g., in \Cref{exp: linear control problem}, then \cref{eq: equilibrium condition third-order I :eq} reduces to
\begin{align*}
0 & < 2 \tilde{u}_{t} ( t,x ) \big\langle \big( 1, \vec{\lambda }^{ t,x, \vec{m} ( t,x ) } ( t,x ) \big), \mathcal{D}_{ \tilde{u} ( t,x ) } \vec{H}^{ t,x, \vec{m} ( t,x ) }_{u} \big( t,x, \tilde{u} ( t,x ) \big) \big\rangle \\
  & \quad - | \tilde{u}_{t} ( t,x ) |^{2} \big\langle \big( 1, \vec{\lambda }^{ t,x, \vec{m} ( t,x ) } ( t,x ) \big), \vec{H}^{ t,x, \vec{m} ( t,x ) }_{uu} \big( t,x, \tilde{u} ( t,x ) \big) \big\rangle.
\end{align*}
while \cref{eq: equilibrium condition third-order II :eq} reduces to
\begin{equation*}
\left\{ \begin{aligned}
0 & \ne | \tilde{u}_{t} ( t,x ) |^{2}, \\
0 & \le \tilde{u}_{t} ( t,x ) \big\langle \big( 1, \vec{\lambda }^{ t,x, \vec{m} ( t,x ) } ( t,x ) \big), \mathcal{D}_{ \tilde{u} ( t,x ) } \vec{H}^{ t,x, \vec{m} ( t,x ) }_{u} \big( t,x, \tilde{u} ( t,x ) \big) \big\rangle.
\end{aligned} \right.
\end{equation*}
\end{remark}

\begin{remark}[necessary conditions for CSES]\label{rem: necessary conditions for CSES}
If $\tilde{u} \in \mathcal{U} \cap {C}^{1,2}$ is a CSES and satisfies the conditions in \Cref{thm: second-order derivative condition for CSES},
then the second-order expansion of $J ( t,x, \tilde{u}^{ t, \varepsilon, \tilde{u} ( t,x ) } ) - J ( t,x, \tilde{u} )$ implies that
the right-hand side of \cref{eq: equilibrium condition second-order I :eq} must be non-negative.
Furthermore, if $\tilde{u} \in \mathcal{U} \cap {C}^{2,4}$ is a CSES and satisfies the conditions in \Cref{thm: third-order derivative condition for CSES},
and the right-hand side of \cref{eq: equilibrium condition second-order I :eq} vanishes at some $( t,x )$,
then the third-order expansion of $J ( t,x, \tilde{u}^{ t, \varepsilon, \tilde{u} ( t,x ) } ) - J ( t,x, \tilde{u} )$ implies that
the right-hand side of \cref{eq: equilibrium condition third-order I :eq} must be non-negative.
\end{remark}

To end this section, we should note that the derived equilibrium conditions up to the third order are not necessarily sufficient for verifying a CSES in general.
For example, for $\tilde{u}_{x} \equiv 0$ on the neighborhood of $( t,x )$ with $\tilde{u}_{t} ( t,x ) = 0$, one obtains $J ( t,x, \tilde{u}^{ t, \varepsilon, \tilde{u} ( t,x ) } ) - J ( t,x, \tilde{u} ) = o ( {\varepsilon }^{3} )$,
implying that the fourth-order derivative equilibrium condition requires investigation.
Interested readers can refer to the proofs of \Cref{lem: higher-order expansion rule} and \Cref{thm: second-order derivative condition for CSES,thm: third-order derivative condition for CSES}
to establish the derivative equilibrium conditions of higher orders.
For the control problem formulated by \cref{eq: linear control problem :eq} with the conditions described in \Cref{exp: linear control problem},
we can find a sufficient condition for the given CNEC $\tilde{u}$ to be a CSES by the third-order derivative equilibrium condition.

\section{\texorpdfstring{\Cref{exp: linear control problem}}{Example 3.3} continuation}
\label{sec: Example continuation}

In this section, in addition to the conditions in \Cref{exp: linear control problem},
we suppose that $( B, D, F ) \in {C}^{1}$ with $B \ne 0$ a.e. on $[ 0,T ]$.
Then, we investigate which condition \cref{eq: derived CNEC :eq} yields a CSES for the control problem formulated by \cref{eq: linear control problem :eq}.
Since the objective function in \cref{eq: linear control problem :eq} implies that
\begin{equation*}
\Phi ( t,x,y, \vec{z} ) = {z}_{1} + \psi \bigg( t, {z}_{2} - {z}_{1}^{2}, \ldots, \sum_{j=0}^{n} \binom{n}{j} (-1)^{j} {z}_{1}^{j} {z}_{n-j} \bigg)
\end{equation*}
is independent of $( x,y )$, we conclude that ${U}^{ s,y, \vec{z} } ( t,x )$ is independent of $( y,t,x )$, and hence,
\begin{equation*}
\vec{H}^{ s,y, \vec{z} } ( t,x,u ) = \big( 0, \vec{m}_{x} ( t,x ) \big) ( {A}_{t} x + {B}_{t} u + {C}_{t} ) + \frac{1}{2} \big( 0, \vec{m}_{xx} ( t,x ) \big) | {D}_{t} u + {F}_{t} |^{2}.
\end{equation*}
Define ${f}^{ s, \vec{z} } := {f}^{ s,y, \vec{z} } ( t,x )$ for the $( y,t,x )$-independent $f = U, \vec{\lambda }, \mu $.
Consequently, for the given CNEC \cref{eq: derived CNEC :eq}, the optimality conditions \cref{eq: first-order derivative condition :eq} and \cref{eq: second-order derivative condition :eq} reduce to
\begin{align}
0 & = \langle \vec{\lambda }^{ t, \vec{m} ( t,x ) }, \vec{m}_{x} ( t,x ) {B}_{t} + \vec{m}_{xx} ( t,x ) {D}_{t} \big( {D}_{t} \tilde{u} ( t,x ) + {F}_{t} \big) \rangle, 
\label{eq: first-order derivative condition ex :eq}%%
\\
0 & > \langle \vec{\lambda }^{ t, \vec{m} ( t,x ) }, \vec{m}_{xx} ( t,x ) \rangle, \notag 
\end{align}
i.e., the system \cref{eq: beta system :eq} for $\beta $ (see \cite[the proof of Theorem 5.2]{Wang-Liu-Bensoussan-Yiu-Wei-2025}),
while the third-order derivative equilibrium condition \cref{eq: equilibrium condition third-order I :eq} reduces to
\begin{align}
0 & < 2 \tilde{u}_{t} ( t,x ) \Big\langle \vec{\lambda }^{ t, \vec{m} ( t,x ) }, 
                                          \mathcal{D}_{ \tilde{u} ( t,x ) } \Big( \vec{m}_{x} ( t,x ) {B}_{t} + \vec{m}_{xx} ( t,x ) {D}_{t} \big( {D}_{t} \tilde{u} ( t,x ) + {F}_{t} \big) \Big) \Big\rangle \notag \\
  & \quad - | \tilde{u}_{t} ( t,x ) |^{2} \langle \vec{\lambda }^{ t, \vec{m} ( t,x ) }, \vec{m}_{xx} ( t,x ) \rangle | {D}_{t} |^{2};
\label{eq: equilibrium condition third-order ex :eq}%%
\end{align}
see also \Cref{rem: state-independent case}.

Corresponding to \cref{eq: derived CNEC :eq}, ${X}^{ t,x, \tilde{u} }_{T}$ has a normal distribution with the variation $\int_{t}^{T} | {\beta }_{s} |^{2} ds$
and the diffeomorphisms $\frac{\partial }{\partial x} {X}^{ t,x, \tilde{u} }_{T} = {e}^{ \int_{t}^{T} {A}_{v} dv }$, implying that
$\frac{\partial }{\partial x} \mathbb{E}_{t} [ ( {X}^{ t,x, \tilde{u} }_{T} )^{j} ] = j \mathbb{E}_{t} [ ( {X}^{ t,x, \tilde{u} }_{T} )^{j-1} ] {e}^{ \int_{t}^{T} {A}_{v} dv }$ for any positive integer $j$.
In particular, according to the chain rule,
\begin{equation}
{e}^{ \int_{t}^{T} {A}_{v} dv } = {\partial }_{x} J ( t,x, \tilde{u} ) = {\partial }_{x} {U}^{ t, \vec{m} ( t,x ) } = \langle \vec{\lambda }^{ t, \vec{m} ( t,x ) }, \vec{m}_{x} ( t,x ) \rangle.
\label{eq: a brief result ex :eq}%%
\end{equation}
Moreover, differentiating both sides of $0 = \mathcal{D}_{ \tilde{u} ( t,x ) } \vec{m} ( t,x )$ with respect to $x$ up to the second order yields
$0 = \mathcal{D}_{ \tilde{u} ( t,x ) } \vec{m}_{x} ( t,x ) + {A}_{t} \vec{m}_{x} ( t,x )$ and $0 = \mathcal{D}_{ \tilde{u} ( t,x ) } \vec{m}_{xx} ( t,x ) + 2 {A}_{t} \vec{m}_{xx} ( t,x )$.
Consequently,
\begin{align*}
& \mathcal{D}_{ \tilde{u} ( t,x ) } \Big( \vec{m}_{x} ( t,x ) {B}_{t} + \vec{m}_{xx} ( t,x ) {D}_{t} \big( {D}_{t} \tilde{u} ( t,x ) + {F}_{t} \big) \Big) \\
& = \vec{m}_{x} ( t,x ) {\partial }_{t} {B}_{t} + \vec{m}_{xx} ( t,x ) \big( {D}_{t} \tilde{u} ( t,x ) + {F}_{t} \big) {\partial }_{t} {D}_{t} 
  + \vec{m}_{xx} ( t,x ) {D}_{t} {\partial }_{t} \big( {D}_{t} \tilde{u} ( t,x ) + {F}_{t} \big) \\
& \quad - {A}_{t} \vec{m}_{x} ( t,x ) {B}_{t} - 2 {A}_{t} \vec{m}_{xx} ( t,x ) {D}_{t} \big( {D}_{t} \tilde{u} ( t,x ) + {F}_{t} \big).
\end{align*}
Plugging this identity back into the right-hand side (RHS for short) of \cref{eq: equilibrium condition third-order ex :eq}
and using \cref{eq: derived CNEC :eq}, \cref{eq: first-order derivative condition ex :eq} and \cref{eq: a brief result ex :eq} to simplify the resulting expression yields
\begin{align*}
\text{RHS of \cref{eq: equilibrium condition third-order ex :eq}}
& = 2 \tilde{u}_{t} ( t,x ) \langle \vec{\lambda }^{ t, \vec{m} ( t,x ) }, \vec{m}_{x} ( t,x ) \rangle ( {A}_{t} {B}_{t} + {\partial }_{t} {B}_{t} ) \\
& \quad + 2 \tilde{u}_{t} ( t,x ) \langle \vec{\lambda }^{ t, \vec{m} ( t,x ) }, \vec{m}_{xx} ( t,x ) \rangle \big( {D}_{t} \tilde{u} ( t,x ) + {F}_{t} \big) {\partial }_{t} {D}_{t} \\
& \quad + 2 \tilde{u}_{t} ( t,x ) \langle \vec{\lambda }^{ t, \vec{m} ( t,x ) }, \vec{m}_{xx} ( t,x ) \rangle {D}_{t} {\partial }_{t} \big( {D}_{t} \tilde{u} ( t,x ) + {F}_{t} \big) \\
& \quad - | \tilde{u}_{t} ( t,x ) |^{2} \langle \vec{\lambda }^{ t, \vec{m} ( t,x ) }, \vec{m}_{xx} ( t,x ) \rangle | {D}_{t} |^{2} \\
& = 2 \tilde{u}_{t} ( t,x ) {B}_{t} {e}^{ \int_{t}^{T} {A}_{v} dv } ( {A}_{t} + {\partial }_{t} \ln | {B}_{t} | )
  - 2 \tilde{u}_{t} ( t,x ) {B}_{t} {e}^{ \int_{t}^{T} {A}_{v} dv } {\partial }_{t} \ln {D}_{t} \\
& \quad - 2 \tilde{u}_{t} ( t,x ) {B}_{t} {e}^{ \int_{t}^{T} {A}_{v} dv } {\partial }_{t} \ln | {D}_{t} \tilde{u} ( t,x ) + {F}_{t} |
        + \frac{ {B}_{t} {D}_{t} | \tilde{u}_{t} ( t,x ) |^{2} }{ {D}_{t} \tilde{u} ( t,x ) + {F}_{t} } {e}^{ \int_{t}^{T} {A}_{v} dv } \\
& = \bigg( \Big| \frac{ {D}_{t} }{ {\beta }_{t} } \tilde{u}_{t} ( t,x ) {e}^{ \int_{t}^{T} {A}_{v} dv } \Big|^{2}
         - 2 {\partial }_{t} \ln \frac{ {\beta }_{t} {D}_{t} }{ {B}_{t} } \cdot \frac{ {D}_{t} }{ {\beta }_{t} } \tilde{u}_{t} ( t,x ) {e}^{ \int_{t}^{T} {A}_{v} dv } \bigg) \frac{ {B}_{t} {\beta }_{t} }{ {D}_{t} },
\end{align*}
where ${B}_{t} {\beta }_{t} > 0$ according to \cref{eq: beta system :eq} unless ${B}_{t} = 0$.
Therefore, for a.e. $t \in [ 0,T ]$, \cref{eq: equilibrium condition third-order ex :eq} is equivalent to
\begin{equation*}
  \bigg| {\partial }_{t} \ln \frac{ {\beta }_{t} {D}_{t} }{ {B}_{t} } \bigg|
< \bigg| \frac{ {D}_{t} }{ {\beta }_{t} } \tilde{u}_{t} ( t,x ) {e}^{ \int_{t}^{T} {A}_{v} dv } - {\partial }_{t} \ln \frac{ {\beta }_{t} {D}_{t} }{ {B}_{t} } \bigg|
\equiv \bigg| {A}_{t} - {\partial }_{t} \ln \frac{ | {D}_{t} |^{2} }{ | {B}_{t} | } - \frac{ {D}_{t} {\partial }_{t} {F}_{t} - {F}_{t} {\partial }_{t} {D}_{t} }{ {\beta }_{t} {D}_{t} } {e}^{ \int_{t}^{T} {A}_{v} dv } \bigg|.
\end{equation*}
That is, the derived CNEC \cref{eq: derived CNEC :eq} satisfying \cref{eq: beta system :eq} is a CSES if the above inequality holds for a.e. $t \in [ 0,T ]$.

In particular, if $\frac{ {F}_{t} }{ {D}_{t} }$ is independent of $t$, then the sufficient condition reduces to
\begin{equation*}
\bigg| {A}_{t} - {\partial }_{t} \ln \frac{ | {D}_{t} |^{2} }{ | {B}_{t} | } \bigg| > \bigg| {\partial }_{t} \ln \frac{ {\beta }_{t} {D}_{t} }{ {B}_{t} } \bigg|,
\end{equation*}
which implies that ${A}_{t}$ is far away from ${\partial }_{t} \ln \frac{ | {D}_{t} |^{2} }{ | {B}_{t} | }$ for all $t \in [ 0,T )$.
On the other hand, for the mean-variance problem with $J ( t,x,u ) = \mathbb{E}_{t} [ {X}^{t,x,u}_{T} ] - \frac{\gamma }{2} \mathbb{E}_{t} [ ( {X}^{u}_{T} - \mathbb{E}_{t} [ {X}^{u}_{T} ] )^{2} ]$, $\gamma > 0$,
one obtains ${\beta }_{t} = \frac{ {B}_{t} }{ \gamma {D}_{t} }$ from \cref{eq: beta system :eq}
and then arrives at the CNEC $\tilde{u} ( t,x ) = \frac{ {B}_{t} }{ \gamma | {D}_{t} |^{2} } {e}^{ - \int_{t}^{T} {A}_{v} dv } - \frac{ {F}_{t} }{ {D}_{t} }$.
Moreover, this CNEC is a CSES if
\begin{equation*}
0 \ne {A}_{t} - {\partial }_{t} \ln \frac{ | {D}_{t} |^{2} }{ | {B}_{t} | } - \frac{ {D}_{t} {\partial }_{t} {F}_{t} - {F}_{t} {\partial }_{t} {D}_{t} }{ {\beta }_{t} {D}_{t} } {e}^{ \int_{t}^{T} {A}_{v} dv }, \quad
a.e. ~ t \in [ 0,T ).
\end{equation*}
In the case where $( B,D,F )$ are constant, this sufficient condition reduces to ${A}_{t} \ne 0$ for a.e. $t \in [ 0,T )$; i.e., the zeros of $A$ on $[ 0,T )$ form a Lebesgue null set.
Furthermore, if $A \equiv 0$ on the right neighborhood of $t$, then $\tilde{u} ( \cdot, x ) \equiv \tilde{u} ( t,x )$, implying that $\tilde{u}^{ t, \varepsilon, \tilde{u} ( t,x ) } = \tilde{u}$.
Thus, the derived mean-variance CNEC $\tilde{u}$ must be a CSES if $( A,B,D,F )$ are constant.

In terms of the mean-variance-skewness-kurtosis problem with 
\begin{equation*}
J ( t,x,u ) = \mathbb{E}_{t} [ {X}^{t,x,u}_{T} ] - \frac{\gamma }{2} \mathbb{E}_{t} \big[ ( {X}^{u}_{T} - \mathbb{E}_{t} [ {X}^{u}_{T} ] )^{2} \big]
            + \frac{\rho }{6} \mathbb{E}_{t} \big[ ( {X}^{u}_{T} - \mathbb{E}_{t} [ {X}^{u}_{T} ] )^{3} \big]
            - \frac{\kappa }{24} \mathbb{E}_{t} \big[ ( {X}^{u}_{T} - \mathbb{E}_{t} [ {X}^{u}_{T} ] )^{4} \big],
\end{equation*}
where $\gamma, \kappa > 0$, \cref{eq: derived CNEC :eq} is a CNEC
if $y = \frac{1}{2} \int_{t}^{T} | {\beta }_{s} |^{2} ds$ solves $\frac{1}{2} \int_{t}^{T} | \frac{ {B}_{s} }{ {D}_{s} } |^{2} ds = \int_{0}^{y} | \gamma + \kappa z |^{2} dz$ and ${B}_{t} {\beta }_{t} \ge 0$
according to \Cref{exp: linear control problem}.
Thus, $\frac{ {\beta }_{t} {D}_{t} }{ {B}_{t} } = ( {\gamma }^{3} + \frac{3}{2} \kappa \int_{t}^{T} | \frac{ {B}_{s} }{ {D}_{s} } |^{2} ds )^{ - \frac{1}{3} }$.
Consequently, $\tilde{u}$ is a CSES if
\begin{equation*}
0 < \bigg( \frac{ 2 {\gamma }^{3} }{\kappa } + 3 \int_{t}^{T} \Big| \frac{ {B}_{s} }{ {D}_{s} } \Big|^{2} ds \bigg)
    \bigg| {A}_{t} - {\partial }_{t} \ln \frac{ | {D}_{t} |^{2} }{ | {B}_{t} | } - \frac{ {D}_{t} {\partial }_{t} {F}_{t} - {F}_{t} {\partial }_{t} {D}_{t} }{ {\beta }_{t} {D}_{t} } {e}^{ \int_{t}^{T} {A}_{v} dv } \bigg|
  - \Big| \frac{ {B}_{t} }{ {D}_{t} } \Big|^{2}.
\end{equation*}
Conversely, according to \Cref{rem: necessary conditions for CSES}, $\tilde{u}$ is a CSES only if the right-hand side of the above inequality is non-negative.
Hence, the derived mean-variance-skewness-kurtosis CNEC $\tilde{u}$ is not necessarily a CSES even though $( A,B,D,F )$ are constant.

\begin{remark}
By taking advantage of the independence of $x$ for \cref{eq: derived CNEC :eq}, one can adopt straightforward calculations to examine whether the derived CNEC $\tilde{u}$ is a CSES.
Indeed, ${X}^{ t,x, \tilde{u}^{ t, \varepsilon, \zeta } }_{T}$ is normally distributed with a mean of
\begin{equation*}
\mathbb{E}_{t} [ {X}^{ t,x, \tilde{u}^{ t, \varepsilon, \zeta } }_{T} ] 
= \mathbb{E}_{t} [ {X}^{ t,x, \tilde{u} }_{T} ]
+ {e}^{ \int_{t}^{T} {A}_{v} dv } \int_{t}^{ t + \varepsilon } {e}^{ - \int_{t}^{s} {A}_{v} dv } {B}_{s} \big( \zeta - \tilde{u} ( s,0 ) \big) ds,
\end{equation*}
a variance of
\begin{align*}
& \mathbb{E}_{t} \big[ \big( {X}^{ t,x, \tilde{u}^{ t, \varepsilon, \zeta } }_{T} - \mathbb{E}_{t} [ {X}^{ t,x, \tilde{u}^{ t, \varepsilon, \zeta } }_{T} ] \big)^{2} \big] \\
& = \mathbb{E}_{t} \big[ \big( {X}^{ t,x, \tilde{u} }_{T} - \mathbb{E}_{t} [ {X}^{ t,x, \tilde{u} }_{T} ] \big)^{2} \big]
  + {e}^{ 2 \int_{t}^{T} {A}_{v} dv } \int_{t}^{ t + \varepsilon } {e}^{ - 2 \int_{t}^{s} {A}_{v} dv } \big( | {D}_{s} \tilde{u} ( t,0 ) + {F}_{s} |^{2} - | {D}_{s} \tilde{u} ( s,0 ) + {F}_{s} |^{2} \big) ds,
\end{align*}
and higher-order central moments of
\begin{equation*}
  \mathbb{E}_{t} \big[ \big( {X}^{ t,x, \tilde{u}^{ t, \varepsilon, \zeta } }_{T} - \mathbb{E}_{t} [ {X}^{ t,x, \tilde{u}^{ t, \varepsilon, \zeta } }_{T} ] \big)^{j} \big]
= {1}_{\{ j \in 2 \mathbb{N} \}} \frac{j!}{ \frac{j}{2} ! }
  \Big| \frac{1}{2} \mathbb{E}_{t} \big[ \big( {X}^{ t,x, \tilde{u}^{ t, \varepsilon, \zeta } }_{T} - \mathbb{E}_{t} [ {X}^{ t,x, \tilde{u}^{ t, \varepsilon, \zeta } }_{T} ] \big)^{2} \big] \Big|^{ \frac{j}{2} }.
\end{equation*}
Then, one can expand $J ( t,x, \tilde{u}^{ t, \varepsilon, \zeta } ) - J ( t,x, \tilde{u} )$ with respect to $\varepsilon $ up to an adequate order.
\end{remark}

\section{Concluding remarks}
\label{sec: Concluding remarks}

We study strong equilibrium strategies for closed-loop time-inconsistent control problems with higher-order moments.
We investigate the expansion of the objective function with respect to the variational factor up to higher orders and then derive the sufficient condition for a CSES.
Furthermore, we find the (semi-)closed-form expression of the sufficient condition for the higher-order moment problem with linear controlled SDEs.
In particular, if the parameters of the controlled SDE are constant, then the derived mean-variance CNEC must be a CSES, whereas the derived mean-variance-skewness-kurtosis CNEC is not necessarily a CSES.

\bibliographystyle{apacite}
\bibliography{references}

\end{document}